\documentclass[a4paper,reqno,12pt]{amsart}

\usepackage{graphicx}
\usepackage{mathrsfs}
\usepackage{amsfonts}
\usepackage{amssymb}
\usepackage{amsmath}
\usepackage{amsthm}
\usepackage{amscd}
\usepackage[all,2cell]{xy}
\usepackage{xcolor}
\usepackage[pagebackref,colorlinks]{hyperref}
\usepackage{enumerate}
\usepackage[normalem]{ulem}
\usepackage{mathrsfs}
\usepackage{lineno}
\usepackage{bm}
\usepackage{bbm}
\usepackage{comment}

\numberwithin{equation}{section}

\theoremstyle{plain}
\newtheorem{thm}{Theorem}[section]
\newtheorem{prop}[thm]{Proposition}

\newtheorem{lemma}[thm]{Lemma}

\theoremstyle{definition}
\newtheorem{deff}[thm]{Definition}
\newtheorem{example}[thm]{Example}
\newtheorem{remark}[thm]{Remark}

\newcommand{\ran}{\bm{r}}
\newcommand{\domr}{\bm{d}}
\newcommand{\G}{\Gamma}

\newcommand{\im}{\operatorname{Im}}
\newcommand{\dom}{\operatorname{Dom}}
\newcommand{\gr}{\operatorname{gr}}
\newcommand{\GIS}{\operatorname{\Gamma\bf{GrIS}}}
\newcommand{\GG}{\operatorname{\Gamma\bf{GrGp}}}
\newcommand{\SG}{\operatorname{\mathsf{S}}}
\newcommand{\GP}{\operatorname{\mathsf{G}}}

\def\a{\alpha}
\def\b{\beta}

\def \Z{\mathbb Z}
\def\-{\text{-}}

\title{Graded Lawson-Stone Duality}

\author{Roozbeh Hazrat}
\address{R.\ Hazrat: Centre for Research in Mathematics and Data Science, Western Sydney University, AUSTRALIA}
\email{\textcolor[rgb]{0.00,0.00,0.84}{r.hazrat@westernsydney.edu.au}}

\author{Zachary Mesyan}
\address{Z.\ Mesyan: Department of Mathematics, University of Colorado, Colorado Springs, CO, 80918, USA} 
\email{\textcolor[rgb]{0.00,0.00,0.84}{zmesyan@uccs.edu}}



\date{\today}

\begin{document}

\begin{abstract} 
The classical Stone duality associates to each Boolean algebra a topological space consisting of ultrafilters. Lawson's generalisation constructs a dual equivalence of categories of Boolean inverse $\land$-semigroups and Hausdorff ample topological groupoids. We further generalise the duality to graded versions of those categories, while allowing larger classes of morphisms, and provide various illustrations.
\end{abstract}

\maketitle

\section{Introduction}

Stone's representation theorem for Boolean algebras~\cite{stone} famously associates to each Boolean algebra a topological space consisting of ultrafilters. That general idea has been extended over the years to construct dual equivalences of various categories, collectively referred to as \emph{Stone dualities}. In particular, (building on the work of Kellendonk~\cite{kellendonk1,kellendonk2}, Paterson~\cite{paterson1}, and Lenz~\cite{Lenz}) Lawson~\cite{Lawson2, Lawson3} established a duality between certain categories of Boolean inverse $\land$-semigroups and Hausdorff ample topological groupoids. (See also~\cite{BEM,exel,GH,LMS}.)

Lawson's duality results have received significant attention, in part because inverse semigroups and topological groupoids have played a fundamental role in constructing widely studied noncommutative algebras and $C^*$-algebras. One typically builds the relevant inverse semigroups and topological groupoids from some combinatorial object, such as a graph. From the resulting inverse semigroups one can then build enveloping algebras~\cite{wehrung}, while from the resulting topological groupoids one can construct convolution algebras (e.g., Steinberg algebras~\cite{ClarkH} and groupoid $C^*$-algebras~\cite{renault}). In certain circumstances these enveloping and convolution algebras are (graded) isomorphic~\cite{rigby}. The aforementioned connections are summarised in the following diagram.
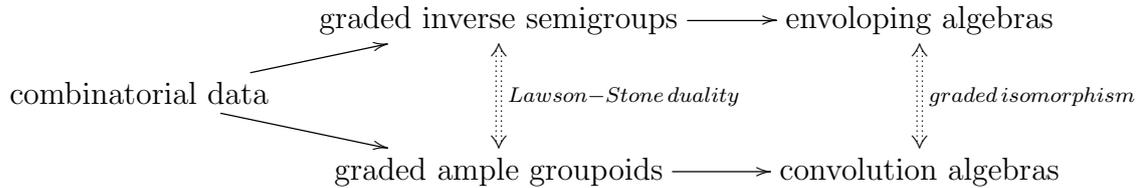
\begin{equation*}\label{schemeobjects}
\xymatrix@=1pc{ 
 & \text{inverse semigroups} \ar@{<:>}[dd]^{Lawson-Stone \, duality} \ar[rr]&& \text{envoloping algebras}  \ar@{<:>}[dd]^{(graded) \, isomorphism}\\
\text{combinatorial data}  \ar[ur] \ar[dr]\\
 & \text{ample groupoids} \ar[rr]&& \text{convolution algebras}}
\end{equation*}

It was noticed a long time ago that the algebras described above are graded, and the grading has played a crucial role in their study. For one thing, the zero-component subalgebras are typically rich and well-behaved. For another, one can lift properties from the homogeneous components to the whole algebras. More concretely, the gradings on Leavitt path algebras (which constitute a class of Steinberg algebras) played a fundamental role in the characterisation of their ideals~\cite{AAS}. Moreover, a central open question in the area is to describe the extent to which graded/equivariant $K$-theory classifies the algebras~\cite{GGG}. 

The gradings of the enveloping and convolution algebras in question arise from the discrete nature of the combinatorial objects from which they are built. For example, in the case of a graph, the algebra grading ultimately results from path lengths. Thus it is natural to study the gradings on the level of semigroups and groupoids, most immediately constructed from the combinatorial objects, rather than on the higher level of resulting algebras. Gradings on both (inverse) semigroups~\cite{RZ} and groupoids~\cite{arasims,chr19} have indeed been systematically explored recently. Moreover, Steinberg~\cite{steinberg}  connected the two, by constructing a graded isomorphism between each graded Hausdorff ample groupoid $\mathscr{G}$ and a certain groupoid constructed from a graded inverse semigroup associated to $\mathscr{G}$. The purpose of this note is to reinforce the connection between graded inverse semigroups and graded topological groupoids, by generalising Lawson's duality to the graded context. We also endow our categories with larger classes of morphisms, which allows us to capture various natural types of morphisms excluded in the earlier context (the details can be found in Definitions~\ref{is-cat-def} and~\ref{gp-cat-def} below). See~\cite{Bice,BS,CG,KL1,KL2,LL,resende} for generalisations of Lawson's duality in other directions, and~\cite{kudryavtseva} for an overview of related developments.

In Sections~\ref{def-sect} and~\ref{basics-sect} we provide definitions of, and basic observations about, our main objects of study, namely graded-Boolean inverse $\land$-semigroups and graded Hausdorff ample groupoids. In Section~\ref{eg-sect} we give several natural constructions of such semigroups, including graded symmetric inverse monoids and distributive completions of graph inverse semigroups. Section~\ref{duality-sect} is devoted to proving our main result (Theorem~\ref{gr-duality-thm}), that the categories of graded-Boolean inverse $\land$-semigroups and graded Hausdorff ample groupoids (with appropriate morphisms) are dually equivalent. Finally, in Section~\ref{rings-sect} we show that given a graded Hausdorff ample groupoid $\mathscr{G}$ and a ring $R$, we have $R\langle \SG(\mathscr{G}) \rangle \cong R\langle \SG^{\gr}(\mathscr{G}) \rangle$, where $\SG^{\gr}(\mathscr{G})$ is the graded inverse semigroup associated to $\mathscr{G}$ by our duality, $\SG(\mathscr{G})$ is its non-graded counterpart from Lawson's duality, and $R\langle \SG(\mathscr{G}) \rangle$ and $R\langle \SG^{\gr}(\mathscr{G}) \rangle$ are corresponding enveloping rings.

\section{Basic Definitions} \label{def-sect}

We begin by recalling and defining the various objects required for a graded version of Lawson's duality, namely, Boolean algebras, gradings, (Boolean) inverse semigroups, and (ample) topological groupoids.

\subsection{Partially Ordered Sets}

Recall that a partially-ordered set $(X, \leq)$ is a \textit{lattice}, if for all $x, y\in X$ there exists in $X$ a meet (i.e., greatest lower bound) $x \wedge y$ and a join (i.e., least upper bound) $x\vee y$. A lattice $X$ is \textit{distributive} if 
\[x\wedge(y\vee z) = (x\wedge y)\vee (x\wedge z) \text{ and } x\vee(y\wedge z) = (x\vee y) \wedge (x\vee z),\] 
for all $x, y, z\in X$. A \emph{Boolean algebra} is a distributive lattice $X$, with top and bottom elements ($1$ and $0$, respectively), along with a unary \emph{complement} operator $\lnot$, satisfying $x \lor \lnot x = 1$ and $x \land \lnot x = 0$, for all $x \in X$.

Let $(X, \leq)$ be a partially ordered set with bottom element $0$, and let $Y \subseteq X$. We set 
\[Y^{\uparrow} := \{x \in X \mid \exists y \in Y \text{ such that } y \leq x\},\]
and
\[Y^{\downarrow} := \{x \in X \mid \exists y \in Y \text{ such that } y \geq x\}.\]
We say that $Y$ is \emph{downward directed} if it is nonempty, and for all $x,y \in Y$ there exists $z\in X$ such that $z \leq x$ and $z \leq y$. If $Y$ is downward directed, $Y = Y^{\uparrow}$, and $0 \notin Y$, then $Y$ is a \emph{filter}. A maximal filter is called an \emph{ultrafilter}. If $Y$ is a filter with the property that $x \in Y$ or $y \in Y$, for all $x,y \in X$ having a join $x \lor y \in X$, with $x \lor y \in Y$, then we say that $Y$ is \emph{prime}. If $Y = Y^{\downarrow}$, then $Y$ is an \emph{order ideal}. An order ideal of the form $\{x\}^{\downarrow}$ ($x \in X$) is called \emph{principal}.

\subsection{Graded Semigroups} \label{gradedsemisubsection}

Recall that a \emph{semigroup} is a nonempty set equipped with an associative binary operation. A \emph{monoid} is a semigroup with an identity element $1$. We assume, throughout, that every semigroup has a zero element $0$, unless stated otherwise. Given a semigroup $S$, we denote by $E(S)$ the set of idempotents of $S$. If for each $s \in S$ there is a unique $s^{-1} \in S$ satisfying $s = ss^{-1}s$ and $s^{-1} = s^{-1}ss^{-1}$, then $S$ is an \emph{inverse semigroup}. In this case $(st)^{-1} = t^{-1}s^{-1}$ for all $s,t \in S$. It is well-known that if $S$ is an inverse semigroup, then the elements of $E(S)$ commute. See~\cite[\S 5]{Howie} or~\cite{Lawson1} for more details about inverse semigroup fundamentals.

Given a semigroup $S$ and a group $\Gamma$, we say that $S$ is $\Gamma$-\emph{graded} if there is a map $\phi : S \setminus\{0\} \to \Gamma$ such that $\phi(st) = \phi(s)\phi(t)$, whenever $st \not= 0$. For each $\alpha \in \Gamma$, we set $S_\alpha:=\phi^{-1}(\alpha) \cup \{0\}$. Equivalently, $S$ is $\Gamma$-\emph{graded} if for each $\alpha \in \Gamma$ there exists $S_\alpha \subseteq S$ such that
\[S = \bigcup_{\alpha \in \G} S_\alpha,\]
where $S_\a S_\b \subseteq S_{\a\b}$ for all $\a,\b \in \Gamma$, and $S_\a \cap S_\b = \{0\}$ for all distinct $\a,\b \in \Gamma$.

Let $S$ be a $\Gamma$-graded semigroup. For each $\a \in \Gamma$ and $s \in S_\alpha \setminus \{0\}$, we say that the \emph{degree} of $s$ is $\alpha$, and write $\deg(s):=\alpha$. If $T \subseteq S$ is such that $T \subseteq S_{\alpha}$ for some $\alpha \in \Gamma$, then we say that $T$ is \emph{homogeneous} (\emph{of degree $\alpha$}). Note that $\deg(s)=\varepsilon$ (the identity element of $\Gamma$) for all $s \in E(S) \setminus \{0\}$, and $\deg(s^{-1}) = \deg(s)^{-1}$ for all $s \in S \setminus \{0\}$, in the case where $S$ is an inverse semigroup. We say that $S$ is \emph{trivially} graded if $S_\varepsilon = S$. It is easy to see that $S_{\varepsilon}$ is a semigroup (with zero), and that $S_\varepsilon$ is an inverse semigroup whenever $S$ is. A homomorphism $\phi:S \rightarrow T$ of $\Gamma$-graded semigroups is \emph{graded} if $\phi(S_\alpha)\subseteq T_\alpha$ for every $\alpha \in \Gamma$. See~\cite{RZ} for more information on graded semigroups.

\subsection{Boolean Inverse Semigroups} \label{bool-def-sect}

Let $S$ be an inverse semigroup. Elements $s,t \in S$ are said to be \emph{compatible} if $s^{-1}t, st^{-1} \in E(S)$ (the set of idempotents), and to be \emph{orthogonal} if $s^{-1}t = 0 = st^{-1}$. We say that subset of $S$ is \emph{compatible} if each pair of elements in that subset is compatible. 

The \emph{natural partial order} $\leq$ on $S$ is defined by $s\leq t$ ($s,t \in S$) if $s=tu$ for some $u \in E(S)$. Equivalently, $s\leq t$ if $s=ut$ for some $u \in E(S)$. In particular, if $u,v \in E(S)$, then $u \leq v$ amounts to $u=uv=vu$. Also if $s \leq t$ for some $s,t \in S$, then $s^{-1} \leq t^{-1}$, and $rs \leq rt$, $sr \leq tr$ for all $r \in S$. See~\cite[\S 5.2]{Howie} for more details.

It is well-known that $E(S)$ forms a lower semilattice (i.e., a partially ordered set where every pair of elements has a meet), relative to the restriction of $\leq$ to $E(S)$, with $e \land f = ef$ for all $e,f \in E(S)$. If each pair of elements in $S$ has a meet, relative to $\leq$, then $S$ is called an \emph{inverse $\land$-semigroup}.

\begin{deff} \label{bool-def}
Let $S$ be a $\Gamma$-graded inverse semigroup, for some group $\Gamma$. Then $S$ is $\Gamma$-\emph{graded-Boolean}, or simply \emph{graded-Boolean}, if it satisfies the following conditions.
\begin{enumerate}[\upshape(1)]
\item For all $\alpha \in \Gamma$, every orthogonal pair $s,t \in S_{\alpha}$ has a join $s\vee t$ in $S$.

\smallskip 

\item $E(S)$ is a distributive lattice, and each nonzero principal order ideal of $E(S)$ is a Boolean algebra.
\end{enumerate}
\end{deff}

\begin{remark}\label{bool-rem1}
In~\cite{Lawson3,Lawson1} an inverse semigroup $S$ is defined to be \emph{Boolean} if it satisfies the following three conditions.
\begin{enumerate}[\upshape(i)]
\item Every nonempty finite compatible subset of $S$ has a join in $S$.

\smallskip 

\item $E(S)$ is a distributive lattice, and each nonzero principal order ideal of $E(S)$ is a Boolean algebra.

\smallskip 

\item For all nonempty compatible $\{s_1, \dots, s_n\} \subseteq S$, and all $s \in S$, the joins $\bigvee_{i=1}^n(s_is)$ and $\bigvee_{i=1}^n(ss_i)$ exist, $\bigvee_{i=1}^n(s_is) = \big(\bigvee_{i=1}^ns_i\big)s$, and $\bigvee_{i=1}^n(ss_i) = s\big(\bigvee_{i=1}^ns_i\big)$.
\end{enumerate}
However, according to~\cite[Proposition 8.2.4]{Lawson1}, an inverse semigroup $S$ is Boolean if and only if it satisfies (ii) and the following condition.
\begin{enumerate}[\upshape(i$'$)]
\item Every orthogonal pair $s,t \in S$ has a join $s\vee t$ in $S$.
\end{enumerate}
Thus Definition~\ref{bool-def} generalises the notion of \emph{Boolean inverse semigroup}. More specifically, if $S$ is a Boolean inverse semigroup and $\Gamma$ is any group, then $S$ is $\Gamma$-graded-Boolean, under the trivial grading.
\end{remark}

\begin{remark}\label{bool-rem2}
The proof of the aforementioned~\cite[Proposition 8.2.4]{Lawson1} is element-wise, and translates verbatim to the graded setting. More specifically, let $S$ be a $\Gamma$-graded inverse semigroup. Then the argument in the proof of~\cite[Proposition 8.2.4]{Lawson1} shows that if $S$ satisfies (2) in Definition~\ref{bool-def}, and every pair of orthogonal elements in $S_{\alpha}$ ($\alpha \in \Gamma$) has a join, then any nonempty finite set of compatible elements of $S_{\alpha}$ has a join, and multiplication distributes over such joins. That is, if $S$ satisfies the conditions in Definition~\ref{bool-def}, then it also satisfies the following.
\begin{enumerate}[\upshape($1'$)]
\item For all $\alpha \in \Gamma$, every nonempty finite compatible subset of $S_{\alpha}$ has a join in $S$.
\end{enumerate}
\begin{enumerate}[\upshape($3'$)]
\item For all $\alpha \in \Gamma$, all nonempty compatible $\{s_1, \dots, s_n\} \subseteq S_{\alpha}$, and all $s \in S$, the joins $\bigvee_{i=1}^n(s_is)$ and $\bigvee_{i=1}^n(ss_i)$ exist, $\bigvee_{i=1}^n(s_is) = \big(\bigvee_{i=1}^ns_i\big)s$, and $\bigvee_{i=1}^n(ss_i) = s\big(\bigvee_{i=1}^ns_i\big)$.
\end{enumerate}
\end{remark}

We are now ready to define the first of the two categories for our duality construction.

\begin{deff} \label{is-cat-def}
Given a group $\Gamma$, we denote by $\GIS$ the category of $\Gamma$-graded-Boolean inverse $\land$-semigroups, having as its morphisms graded homomorphisms $S \to T$ of inverse semigroups, which preserve meets, and restrict to lattice homomorphisms $E(S) \to E(T)$ that preserve complements in principal order ideals.
\end{deff}

A homomorphism $\phi : S \to T$ of inverse semigroups is called \emph{proper} if $\phi^{-1}(X)$ is an ultrafilter on $S$ for every ultrafilter $X$ on $T$. In the duality established in~\cite{Lawson2, Lawson3} the (non-graded) morphisms in the inverse semigroup category are assumed to be proper, in addition to having the properties listed in Definition~\ref{is-cat-def} (i.e., preserving meets, and restricting to idempotent lattice homomorphisms that preserve complements in principal order ideals). We work with larger classes of morphisms in our categories, but will discuss proper morphisms as well. See Example~\ref{morph-eg} for an illustration of the possibilities allowed by relaxing the restrictions on the morphisms.

\subsection{Graded Topological Groupoids} \label{groupoidsubsection}

A \emph{groupoid} $\mathscr{G}$ is a small category in which every morphism is invertible. We denote the set of objects of $\mathscr{G}$, also known as the \emph{unit space}, by $\mathscr{G}^{(0)}$, and we identify these objects with the corresponding identity morphisms. For each morphism $x \in \mathscr{G}$, the object $\domr(x):=x^{-1}x$ is the \emph{domain} of $x$, and the object $\ran(x):=xx^{-1}$ is its \emph{range}. (Thus, two morphisms $x, y \in \mathscr{G}$ are composable as $xy$ if and only if $\domr(x)=\ran(y)$.) Define 
\[\mathscr{G}^{(2)} := \big \{(x,y) \in \mathscr{G} \times \mathscr{G} \mid \domr(x) = \ran(y) \big\}.\]
For subsets $X,Y\subseteq \mathscr{G}$ of morphisms, we define
\begin{equation}\label{pofgtryhf}
XY:=\big \{xy \mid x \in X, y \in Y, \domr(x)=\ran(y) \big\},
\end{equation}
and
\begin{equation}\label{pofgtryhf2}
X^{-1}:=\big \{x^{-1} \mid x \in X \big \}.
\end{equation}

A \textit{topological groupoid} is a groupoid (whose set of morphisms is) equipped with a topology making inversion and composition continuous. A topological groupoid $\mathscr{G}$ is \textit{\'etale} if $\domr: \mathscr{G} \to \mathscr{G}^{(0)}$ is a local homeomorphism, where $\mathscr{G}^{(0)}$ is viewed as a topological space in the topology induced by that on $\mathscr{G}$. (I.e., for all $x \in \mathscr{G}$ there is an open $X \subseteq \mathscr{G}$ such that $x \in X$, $\domr(X)$ is open, and the restriction $\domr|_{X} : X \to \domr(X)$ is a homeomorphism with respect to the relevant subspace topologies.) An open subset $X$ of an \'etale groupoid $\mathscr{G}$ is called a \emph{slice} or \emph{local bisection} if the restrictions $\domr|_{X}$ and $\ran|_{X}$ are injective. An \'etale groupoid $\mathscr{G}$ is \emph{ample} or \emph{Boolean} if the compact slices form a base for its topology, and $\mathscr{G}^{(0)}$ is Hausdorff in the induced topology.

Given a topological groupoid $\mathscr{G}$ and a group $\Gamma$, we say that $\mathscr{G}$ is $\Gamma$-\emph{graded} if there is a continuous functor $\phi:\mathscr{G} \rightarrow \Gamma$, where $\Gamma$ is endowed with the discrete topology. (Such a functor can be viewed as a continuous function from the set of all morphisms of $\mathscr{G}$ to $\Gamma$, such that $\phi(x)\phi(y) = \phi(xy)$ for all $(x,y) \in \mathscr{G}^{(2)}$.) Setting $\mathscr{G}_\alpha:=\phi^{-1}(\alpha)$ for each $\alpha \in \Gamma$, we have
\[\mathscr{G}= \bigcup_{\alpha \in \G} \mathscr{G}_\alpha,\]
where $\mathscr{G}_\alpha \mathscr{G}_\beta \subseteq \mathscr{G}_{\alpha\beta}$ for all $\alpha,\beta \in \Gamma$, and $\mathscr{G}_\alpha \cap \mathscr{G}_\beta = \emptyset$ for all distinct $\alpha,\beta \in \Gamma$. Note that $\mathscr{G}^{(0)} \subseteq \mathscr{G}_\varepsilon$. As with semigroups, we say that each $x \in \mathscr{G}_\alpha$ has \emph{degree} $\alpha$, and write $\deg(x) := \alpha$.

Let $\mathscr{G}$ and $\mathscr{H}$ be topological groupoids, and let $\phi : \mathscr{G} \to \mathscr{H}$ be a \emph{partial} (or \emph{partially defined}) \emph{functor}--that is, $\phi : \mathscr{G}' \to \mathscr{H}$ is a functor, where $\mathscr{G}'$ is an open subgroupoid of $\mathscr{G}$ (i.e., $\mathscr{G}'$ is a subcategory of $\mathscr{G}$, closed under morphism inversion, whose set of morphisms is open in the topology on $\mathscr{G}$). Then $\phi$ is \emph{covering} if it is \emph{star-injective} (i.e., if $\phi(x) = \phi(y)$ and $\domr(x) = \domr(y)$, then $x = y$), and \emph{star-surjective} (i.e., if $\domr(y) = \phi(e)$ for an identity morphism $e$, then $\domr(x) = e$ and $\phi(x) = y$ for some $x \in \mathscr{G}'$). The partial functor $\phi$ is \emph{proper} if it is continuous (relative to the subspace topology on its domain $\mathscr{G}'$), and the inverse image under $\phi$ of every compact subset of $\mathscr{H}$ is a compact subset of $\mathscr{G}'$ (and hence also of $\mathscr{G}$). If $\mathscr{G}$ and $\mathscr{H}$ are $\Gamma$-graded, for some group $\Gamma$, then $\phi$ is \emph{graded} if $\phi(\mathscr{G}_{\alpha} \cap \mathscr{G}') \subseteq \mathscr{H}_{\alpha}$ for all $\alpha \in \Gamma$.

Let $\phi : \mathscr{G}_1 \to \mathscr{G}_2$ and $\psi : \mathscr{G}_2 \to \mathscr{G}_3$ be continuous partial functors of topological groupoids, with domains $\mathscr{G}'_1$ and $\mathscr{G}'_2$, respectively. Then we obtain a continuous partial functor $\psi \circ \phi : \mathscr{G}_1 \to \mathscr{G}_3$ by restricting to the (open) domain $\phi^{-1}(\mathscr{G}'_2) = \phi^{-1}(\mathscr{G}'_2 \cap \phi(\mathscr{G}'_1))$. It is easy to check that if $\phi$ and $\psi$ are covering, proper, or graded, then so is $\psi \circ \phi$, respectively. With that in mind, we are ready to define our second category of interest.

\begin{deff} \label{gp-cat-def}
Given a group $\Gamma$, we denote by $\GG$ the category of $\Gamma$-graded Hausdorff ample groupoids, with graded proper covering partial functors as its morphisms.
\end{deff}

In the duality established in~\cite{Lawson2, Lawson3} the (non-graded) morphisms in the Hausdorff ample groupoid category are assumed to be fully-defined, in addition to being proper and covering. Our duality requires a more permissive notion of ``morphism", in light of our choice of morphisms in the inverse semigroup category (see Definition~\ref{is-cat-def} and subsequent comment).

\section{Graded-Boolean Basics}\label{basics-sect}

In this section we record some preliminary observations about graded-Boolean inverse $\land$-semigroups, that will be used extensively.

\begin{lemma} \label{filter-lemma}
Let $S$ be a $\, \Gamma$-graded inverse semigroup.
\begin{enumerate}[\upshape(1)]
\item If $s,t \in S$ are compatible but not orthogonal, then $\, \deg(s)=\deg(t)$.

\smallskip 

\item If $s,t \in S \setminus \{0\}$ are comparable in the natural partial order, then $\, \deg(s)=\deg(t)$.

\smallskip 

\item If $X \subseteq S$ is a filter, then $\, \deg(s)=\deg(t)$ for all $s, t\in X$.

\smallskip

\item If $\alpha \in \Gamma$, and $s \in S$ is the join or meet for some nonempty subset of $S_{\alpha}$ $($in the natural partial order on $S$$)$, then $s \in S_{\alpha}$. 
\end{enumerate}
\end{lemma}

\begin{proof}
(1) Suppose that $s,t \in S$ are compatible but not orthogonal. Then $s^{-1}t \in E(S) \setminus \{0\}$, and so 
\[\deg(s)^{-1}\deg(t) = \deg(s^{-1}t) = \varepsilon,\]
which implies that $\deg(s)=\deg(t)$.

\smallskip

(2) Taking comparable $s,t \in S \setminus \{0\}$, we may assume that $s=tu$ for some $u \in E(S)$. Since $s \neq 0$, we have
\[\deg(s) = \deg(t)\deg(u) = \deg(t).\]
The statement also follows from (1).

\smallskip

(3) Suppose that $s, t\in X$. Since $X$ is downward directed, $r \leq s$ and $r \leq t$ for some $r \in X$. Since $r,s,t \neq 0$, as elements of a filter, the statement follows from (2).

\smallskip

(4) Suppose that $s \in S$ is the join or meet for some nonempty $T \subseteq S_{\alpha}$. Since $0 \in S_{\alpha}$, we may assume that $s \neq 0$, and hence also that $T \neq \{0\}$. Then $t \leq s$ or $s \leq t$ for some $t \in T\setminus \{0\} \subseteq S_{\alpha} \setminus \{0\}$, and so $s \in S_{\alpha}$, by (2).
\end{proof}

\begin{lemma}[cf.\ \cite{Lawson2}, Lemma 2.2(2,3)] \label{comp-lem}
Let $S$ be an inverse $\land$-semigroup that satisfies condition $\, (2)$ in Definition~\ref{bool-def}, and let $s,t \in S$.  
\begin{enumerate}[\upshape(1)]
\item The join of a pair of elements of a principal order ideal of $E(S)$ is also its join in $S$.

\smallskip 

\item If $s\leq t$, then there is an element $t \setminus s$ of $S$ such that $s$ is orthogonal to $t \setminus s$, and $t = s \lor (t \setminus s)$.

\smallskip 

\item If $s\not\leq t$, then there exists $r \in S\setminus \{0\}$ such that $r \leq s$ and $r\land t = 0$.
\end{enumerate}
\end{lemma}

\begin{proof}
(1) Let $X$ be a principal order ideal of $E(S)$, let $u,v \in X$, and suppose that $w$ is the join of $u$ and $v$ in $X$. Now, suppose that $r \in S$ is such that $u, v \leq r$, and let $p := w \land r$, using the assumption that $S$ is an inverse $\land$-semigroup. Since $u, v \leq r$ and $u, v \leq w$, we have $u, v \leq p$. Since $p \leq w$ and $w \in E(S)$, necessarily $p \in E(S)$. Since $w \in X$, and $X = X^{\downarrow}$, also $p \in X$, and hence $p = w$. Thus $w \leq r$, from which it follows that $u \lor v = w$ in $S$.

\smallskip

(2) Suppose that $s\leq t$, and write $s = tv$, for some $v \in E(S)$. We may assume that $t \neq 0$, since otherwise $s=0$, and we can take $t\setminus s = 0$.

Let $X = \{t^{-1}t\}^{\downarrow} \subseteq E(S)$ be the principal order ideal of $E(S)$ generated by $t^{-1}t$. Then, by hypothesis, $X$ is a Boolean algebra. Since $E(S)$ is commutative, we have
\[s^{-1}st^{-1}t = vt^{-1}tvt^{-1}t = vt^{-1}tv = s^{-1}s,\] 
and so $s^{-1}s \leq t^{-1}t$. Thus $s^{-1}s \in X$, and therefore there exists $u \in X$ such that $s^{-1}s \lor u = t^{-1}t$ and $s^{-1}s \land u = 0$ in $X$. Note that, by (1), $s^{-1}s \lor u = t^{-1}t$ as elements of $S$. Setting $t \setminus s := tu$, we claim that this element has the desired properties.

Since $0 = s^{-1}s \land u  = us^{-1}s$, we have $us^{-1} = us^{-1}ss^{-1} = 0$, and hence also $su = 0$. Thus $(t\setminus s)s^{-1} = tus^{-1} = 0$, and 
\[(t\setminus s)^{-1}s = ut^{-1}s = ut^{-1}tv = t^{-1}tvu = t^{-1}su = 0.\]
Therefore $s$ and $t \setminus s$ are orthogonal.

Since $t \setminus s = tu \leq t$, to prove that $t = s \lor (t \setminus s)$, it suffices to take $p \in S$ such that $s, t \setminus s \leq p$, and show that $t \leq p$. We have 
\[t^{-1}s = t^{-1}tv = vt^{-1}tv = s^{-1}s,\]
and $t^{-1}(t \setminus s) = t^{-1}tu = u$. Since $t^{-1}s \leq t^{-1}p$ and $t^{-1}(t \setminus s) \leq t^{-1}p$, it follows that
\[t^{-1}t = s^{-1}s \lor u = t^{-1}s \lor t^{-1}(t \setminus s) \leq t^{-1}p,\]
and therefore $t = tt^{-1}t \leq tt^{-1}p \leq p$, as desired.

\smallskip

(3) Suppose that $s\not\leq t$. Since $S$ is an inverse $\land$-semigroup, $s \land t \in S$, and since $s \land t \leq s$, we may form $r := s\setminus (s\land t)$, by (2), where $r \leq s$ and $s = (s\land t) \lor r$. It cannot be the case that $r = 0$, since then we would have $s = s\land t$, and so $s \leq t$, contrary to hypothesis. Next, note that since, by (2), $r$ and $s \land t$ are orthogonal, it must be the case that $r\land (s\land t) = 0$. (Given orthogonal $p,q \in S$, setting $x := p \land q$, and writing $x=pu=qv$ for $u,v \in E(S)$, we have $p^{-1}pu = p^{-1}qv = 0$. Thus $x = pu = pp^{-1}pu = 0$.) Therefore, using the fact that $r \leq s$, we conclude that $0= (r\land s)\land t = r\land t$.
\end{proof}

The next lemma collects some basic well-known facts about filters on inverse semigroups. We give proofs, in cases where we could not locate them in the literature, for the convenience of the reader.

\begin{lemma} \label{idempt-filt-lem}
Let $S$ be an inverse semigroup, and let $X$ be a filter on $S$.
\begin{enumerate}[\upshape(1)]
\item If $s \in S$ is such that $\, 0 \notin sX$, then $(sX)^{\uparrow}$ is a filter on $S$.

\smallskip 

\item If $s,t \in S$ are such that $s^{-1}t \in X$, then $\, 0 \notin sX$.

\smallskip 

\item $Y = (X^{-1} X)^{\uparrow}$ is a filter on and an inverse subsemigroup of $S$ $($without $\, 0$$)$, and $X = (sY)^{\uparrow}$ for all $s \in X$.

\smallskip 

\item If $s \in S$ is such that $s^{-1}s \in X$ and $X$ is an ultrafilter, then $(sX)^{\uparrow}$ is an ultrafilter on $S$.

\smallskip 

\item If $X$ is an ultrafilter, then $(X^{-1} X)^{\uparrow}$ is an ultrafilter on $S$, and $E((X^{-1} X)^{\uparrow})$ is an ultrafilter on $E(S)$.
\end{enumerate}
\end{lemma}

\begin{proof}
(1) Let $s \in S$ be such that $0 \notin sX$, and let $r,t \in X$. Since $X$ is a filter, there exists $p \in X$ such that $p \leq r$ and $p \leq t$, and so $sp \leq sr$ and $sp \leq st$. Since $sp \in sX$, this shows that $sX$ is downward directed, which implies that $(sX)^{\uparrow}$ is downward directed as well. Given that $0 \notin sX$, it follows that $(sX)^{\uparrow}$ is a filter.

\smallskip

(2) Suppose that $s,t \in S$ are such that $s^{-1}t \in X$, and let $p \in X$ be arbitrary. Then there exists $r \in X$ such that $r \leq s^{-1}t$ and $r \leq p$. Write $r = s^{-1}tu = pv$ for some $u, v \in E(S)$. Then $r = s^{-1}ss^{-1}tu = s^{-1}spv$. Since $r \neq 0$, being an element of the filter $X$, it follows that $sp \neq 0$, and so $0 \notin sX$.

\smallskip

(3) This is~\cite[Lemma 2.8(1)]{Lawson3}.

\smallskip

(4) Suppose that $s \in S$, $s^{-1}s \in X$, and $X$ is an ultrafilter. Then $0 \notin sX$, by (2), and so $(sX)^{\uparrow}$ is a filter, by (1). To show that $(sX)^{\uparrow}$ is an ultrafilter, suppose that $(sX)^{\uparrow} \subseteq Z$ for some filter $Z$ on $S$. Since $ss^{-1}s \in Z$, we have $0 \notin s^{-1}Z$, by (2), and so $(s^{-1}Z)^{\uparrow}$ is a filter, by (1). Now, $s^{-1}sX \subseteq s^{-1}Z$, and so $(s^{-1}sX)^{\uparrow} \subseteq (s^{-1}Z)^{\uparrow}$. According to~\cite[Lemma 2.8(2)]{Lawson3}, since $X$ contains an idempotent, $X = (XX)^{\uparrow}$, and $X$ is an inverse subsemigroup of $S$ (without $0$). Thus, $X = (X^{-1}X)^{\uparrow}$, and so $(s^{-1}sX)^{\uparrow} = X$, by (3). Given that $X$ is an ultrafilter, we conclude that $X = (s^{-1}Z)^{\uparrow}$. Finally, letting $t \in Z$ be any element, we have $s^{-1}t \in X$, and so $t \geq ss^{-1}t \in sX$. Thus $(sX)^{\uparrow} = Z$, and so $(sX)^{\uparrow}$ is an ultrafilter.

\smallskip

(5) Suppose that $X$ is an ultrafilter, and write $Y = (X^{-1} X)^{\uparrow}$. Then $Y$ is a filter, by (3). Now suppose that $Y \subseteq Z$ for some filter $Z$ on $S$. Taking any $r \in X$, we have $X = (rY)^{\uparrow} \subseteq (rZ)^{\uparrow}$, by (3). Since $r^{-1}r \in Y \subseteq Z$, (2) implies that $0 \notin rZ$, and so $(rZ)^{\uparrow}$ is a filter, by (1). Given that $X$ is an ultrafilter, $(rY)^{\uparrow} = (rZ)^{\uparrow}$. Then letting $t \in Z$ be arbitrary, we have $rt \geq rs$ for some $s \in Y$, and so $t \geq r^{-1}rt \geq r^{-1}rs$. Since $Y$ is semigroup, by (3), and $r^{-1}r \in Y$, we have $r^{-1}rs \in Y$. Since $Y$ is a filter, it follows that $t \in Y$, and so we conclude that $Y=Z$. Therefore $Y$ is an ultrafilter on $S$.

$E((X^{-1} X)^{\uparrow})$ being an ultrafilter on $E(S)$ is a special case of~\cite[Lemma 2.18(3)]{Lawson3}. 
\end{proof}

Next, we record some concepts, first introduced in~\cite{Lenz}, and used extensively in~\cite{Lawson3}. Let $S$ be an inverse $\land$-semigroup. Given $s,t \in S$, write $s \to t$ if $r \land t \neq 0$ for all $r \in S \setminus \{0\}$ satisfying $r \leq s$. If $s \to t$ and $t \to s$, then we write $s \leftrightarrow t$. It is shown in~\cite[Lemma 2.1]{Lawson3} that $\leftrightarrow$ is a $0$-restricted congruence on $S$, and so one can form the quotient semigroup $S/\leftrightarrow$. (Here, \emph{$0$-restricted} refers to the property that for all $s \in S$, we have $s\leftrightarrow 0$ if and only if $s=0$.) Note that if $s\leq t$ for some $s,t \in S$, then $r \land t = r$ for all $r \in S \setminus \{0\}$ satisfying $r \leq s$, and so $s \to t$. The semigroup $S$ is said to be \emph{separative} provided that $s \leq t$ if and only if $s \to t$, for all $s,t \in S$, in which case $\leftrightarrow$ is simply equality. According to~\cite[Proposition 2.4]{Lawson3}, $S/\leftrightarrow$ is separative.

More generally, given $s_1, \dots, s_m, t_1, \dots, t_n \in S$ write $\{s_1, \dots, s_m\} \to \{t_1, \dots, t_n\}$ if for each $1 \leq i \leq m$ and $r \in S\setminus \{0\}$ satisfying $r \leq s_i$, there exists $1 \leq j \leq n$ such that $r \land t_j \neq 0$. The notation $\{s_1, \dots, s_m\} \leftrightarrow \{t_1, \dots, t_n\}$ is defined in the obvious way.

\begin{lemma}[cf.\ \cite{Lawson1}, Proposition 9.4.8] \label{ultrafilt-lem}
Let $S$ be a $\, \Gamma$-graded-Boolean inverse $\land$-semigroup, and let $X$ be a filter on $S$. Then $X$ is an ultrafilter if and only if $X$ is a prime filter.
\end{lemma}

\begin{proof}
Suppose that $X$ is an ultrafilter, let $Y = (X^{-1} X)^{\uparrow}$, and suppose that $r := s \lor t \in X$, for some $s,t \in S$. Then $r^{-1}r = (s^{-1}s) \lor (t^{-1}t)$, by \cite[Lemma 2.4.4(1)]{Lawson1}. Since $r^{-1}r \in Y$, we have $(s^{-1}s) \lor (t^{-1}t) \in E(Y)$. Note that $Y$ is an ultrafilter on $S$, and $E(Y)$ is an ultrafilter on $E(S)$, by Lemma~\ref{idempt-filt-lem}(5). 

According to \cite[Lemma 2.19(1)]{Lawson3}, since $E(Y)$ is an ultrafilter on the inverse semigroup $E(S)$, it is a \emph{tight} filter (i.e., if $x \in E(Y)$ and $\{x\} \to \{y_1, \dots, y_n\}$ for some $y_1, \dots, y_n \in E(S)$, then $y_i \in E(Y)$ for some $i$), and, by \cite[Lemma 2.41]{Lawson3}, $E(Y)$ is therefore a prime filter. Thus $s^{-1}s \in E(Y)$ or $t^{-1}t \in E(Y)$. We may assume that $s^{-1}s \in E(Y) \subseteq Y$, and let $Z = (sY)^{\uparrow}$. Then $0 \notin sY$, by Lemma~\ref{idempt-filt-lem}(2), and so $Z$ is a filter on $S$, by Lemma~\ref{idempt-filt-lem}(1). Now, for all $p \in Y$ we have $rp \geq sp$, from which it follows that $(rY)^{\uparrow} \subseteq (sY)^{\uparrow} = Z$. But $X = (rY)^{\uparrow}$, by Lemma~\ref{idempt-filt-lem}(3), and so $(rY)^{\uparrow} = (sY)^{\uparrow}$, since $X$ is an ultrafilter. In particular, $s = ss^{-1}s \in sY \subseteq X$, and so $X$ is prime.

For the converse, suppose that $X$ is prime. Since $S$ is a lower semilattice, according to \cite[Lemma 9.1.2]{Lawson1}, to conclude that $X$ is an ultrafilter, it suffices to show that if $s\in S$ is such that $s \land t \neq 0$ for all $t \in X$, then $s \in X$.

First, suppose that $s \in S \setminus  X$ and $t \in X$ are such that $s \leq t$. We claim that $s \land r = 0$ for some $r \in X$. By Lemma~\ref{comp-lem}(2), there exists $t \setminus s \in S$ such that $s$ is orthogonal to $t \setminus s$ and $t = s \lor (t \setminus s)$. Since $s \lor (t \setminus s) \in X$, the filter $X$ is prime, and $s \notin X$, we have $t \setminus s \in X$. But since $s$ and $t \setminus s$ are orthogonal, $s \land (t \setminus s) = 0$ (see the proof of Lemma~\ref{comp-lem}(3) for more details on this standard fact), as claimed.

Now, let $s \in S$ be such that $s \land t \neq 0$ for all $t \in X$. Let $t \in X$ be arbitrary, and set $p := s \land t$. Suppose that $p \notin X$. Since $p \leq t$, there exists $r \in X$ such that $p \land r = 0$, by the previous paragraph. Hence
\[0 = p \land r =  (s \land t) \land r = s \land (t \land r).\]
Since $X = X^{\uparrow}$ and $X$ is downward directed, $t \land r \in X$, which contradicts our choice of $s$. Therefore $p \in X$, and since $p \leq s$, we have $s \in X$. It follows that $X$ is an ultrafilter.
\end{proof}

Given an inverse $\land$-semigroup $S$, let $\mathsf{FC}(S)$ denote the set of all finitely-generated compatible (as subsets of $S$) order ideals of $S$. It is shown in~\cite[Proposition 2.5]{Lawson3} that $\mathsf{FC}(S)$ is an inverse semigroup (with the obvious subset multiplication operation). Let $I, J \in \mathsf{FC}(S)$, and let $\{s_1, \dots, s_m\}$ and $\{t_1, \dots, t_n\}$ be generating sets for $I$ and $J$, respectively, as order ideals (so $I = \{s_1, \dots, s_m\}^{\downarrow}$ and $J = \{t_1, \dots, t_n\}^{\downarrow}$). Write $I \equiv J$ if $\{s_1, \dots, s_m\} \leftrightarrow \{t_1, \dots, t_n\}$. It is shown in~\cite[Lemma 2.14]{Lawson3} that $\equiv$ is a congruence on $\mathsf{FC}(S)$, and so one can form the quotient semigroup $\mathsf{FC}(S)/\equiv$. Now, define the \emph{distributive completion} of $S$ to be $\mathsf{D}(S) := \mathsf{FC}(S/\leftrightarrow)/\equiv$, and define $S$ to be \emph{pre-Boolean} if $\mathsf{D}(S)$ is Boolean.

\begin{prop} \label{pre-bool}
Every graded-Boolean inverse $\land$-semigroup is pre-Boolean and separative.
\end{prop}

\begin{proof}
Let $S$ be a $\Gamma$-graded-Boolean inverse $\land$-semigroup. It follows immediately from Definition~\ref{bool-def} that $E(S)$ is a Boolean inverse $\land$-semigroup. It is shown in the proof of the Duality Theorem~\cite[Section 2.4]{Lawson3} that every Boolean inverse $\land$-semigroup is isomorphic to its distributive completion, and so $E(S) \cong \mathsf{D}(E(S))$. Thus $\mathsf{D}(E(S))$ is Boolean, and so $E(S)$ is pre-Boolean. It follows that $S$ is pre-Boolean, by~\cite[Theorem 1.5(1)]{Lawson3}, which says that an inverse $\land$-semigroup is pre-Boolean if and only if its semilattice of idempotents is.

To prove that $S$ is separative, let $s,t \in S$. As noted above, $s \leq t$ implies that $s \to t$. For the other direction, suppose that $s \to t$ and $s \not\leq t$. Then, by Lemma~\ref{comp-lem}(3), there exists $r \in S\setminus \{0\}$ such that $r \leq s$ and $r\land t = 0$, contradicting $s \to t$. Thus $s \to t$ implies that $s \leq t$, and so $s \leq t$ if and only if $s \to t$.
\end{proof}

\section{Examples of Graded-Boolean Inverse Semigroups} \label{eg-sect}

Let us next discuss a few ways of building graded-Boolean inverse $\land$-semigroups. By Remark~\ref{bool-rem1}, every Boolean inverse $\land$-semigroup is graded-Boolean. In Proposition~\ref{gr-sym-inv-semgp} below we construct graded-Boolean $\land$-inverse semigroups that are not Boolean. To state that result we first recall (graded) symmetric inverse monoids.

Let $X$ be a nonempty set. For any $A,B \subseteq X$, a bijective function $\phi: A\rightarrow B$ is called a \emph{partial symmetry} of $X$. Here we let $\dom(\phi):=A$ and $\im(\phi): = B$. The set $\mathcal{I}(X)$ of all partial symmetries of $X$ is an inverse semigroup, with respect to composition of relations, known as the \emph{symmetric inverse monoid}. Specifically, for all $\phi,\psi \in \mathcal{I}(X)$, $\phi \psi$ is taken to be the composite of $\phi$ and $\psi$ as functions, restricted to the domain $\psi^{-1}(\im(\psi)\cap \dom(\phi))$. The empty function plays the role of the zero element in $\mathcal{I}(X)$. It is easy to see that $\phi \in \mathcal{I}(X)$ is idempotent if and only if $\dom(\phi)=\im(\phi)$ and $\phi(x) = x$ for all $x \in \dom(\phi)$. Moreover, for all $\phi, \psi \in \mathcal{I}(X)$ we have  $\phi \leq \psi$ if and only if $\dom(\phi) \subseteq \dom(\psi)$, and $\phi(x)=\psi(x)$ for all $x \in \dom(\phi)$. 

It is noted in~\cite[Example 2.25]{Lawson2} that $\mathcal{I}(X)$ is a Boolean inverse $\land$-semigroup for all $X$. The only possible grading on $\mathcal{I}(X)$ is the trivial one, provided $|X| \geq 3$, where $|X|$ denotes the cardinality of $X$ (see~\cite[Proposition 5.1]{RZ}), and so $\mathcal{I}(X)$ is generally not graded-Boolean in any interesting way. There are, however, natural graded-Boolean subsemigroups of $\mathcal{I}(X)$, which we describe next.

Let $X$ be a set and $\Gamma$ a group. We refer to any function $\phi: X \to \Gamma$ as a $\Gamma$-\emph{grading on} $X$, and refer to $X$, together with a specific grading, as a $\Gamma$-\emph{graded} set. Now fix a grading $\phi: X \to \Gamma$. If $\phi(x) = \phi(y)$ for all $x,y \in X$, then we say that the grading $\phi$ is \emph{trivial}. For each $\alpha \in \Gamma$ let $X_{\alpha} := \phi^{-1}(\alpha)$, and for each nonempty $Y \subseteq X$ write $Y_\alpha:= Y\cap X_\alpha$. Also, for each $\alpha \in \Gamma$ let 
\begin{equation}\label{mothercomp} 
\mathcal{I}(X)_{\alpha}:= \big \{\phi \in \mathcal{I}(X) \, \big \vert \, \phi(\dom(\phi)_{\beta}) \subseteq X_{\alpha\beta} \text{ for all } \beta \in \Gamma \big \}, 
\end{equation} 
and set 
\begin{equation}\label{mothersemigroup} 
\mathcal{I}^{\gr}(X):= \bigcup _{\alpha \in \Gamma} \mathcal{I}(X)_{\alpha}. 
\end{equation} 
According to~\cite[Proposition 5.3]{RZ}, $\mathcal{I}^{\gr}(X)$ is a $\Gamma$-graded inverse semigroup, with respect to the grading given in (\ref{mothercomp}). Moreover, it is shown in~\cite[Proposition 5.4]{RZ} that every $\Gamma$-graded inverse semigroup can be embedded in $\mathcal{I}^{\gr}(X)$, for some $\Gamma$-graded set $X$.

\begin{prop} \label{gr-sym-inv-semgp}
Let $X$ be a nonempty $\, \Gamma$-graded set. 
\begin{enumerate}[\upshape(1)]
\item $\mathcal{I}^{\gr}(X)$ is a $\, \Gamma$-graded-Boolean inverse $\land$-semigroup.

\smallskip 

\item If $\, |X| \geq 3$, then $\mathcal{I}^{\gr}(X)$ is Boolean if and only if the grading on $X$ is trivial.
\end{enumerate}
\end{prop}

\begin{proof}
(1) Let $\phi, \psi \in \mathcal{I}^{\gr}(X)$ be nonzero, let 
\[Y := \big\{x \in \dom(\phi) \cap \dom(\psi) \mid \phi(x) = \psi(x)\big\},\]
and let $\chi \in \mathcal{I}(X)$ be such that $\dom(\chi) = Y = \im(\chi)$ and $\chi(x) = \phi(x) = \psi(x)$ for all $x \in Y$. Then clearly $\chi = \phi \land \psi$. Moreover, if $Y \neq \emptyset$ (in which case $\chi \neq 0$), then necessarily $\deg(\phi) = \deg(\psi)$ and $\chi \in \mathcal{I}(X)_{\deg(\phi)} \subseteq \mathcal{I}^{\gr}(X)$ (e.g., by Lemma~\ref{filter-lemma}(2)). It follows that $\mathcal{I}^{\gr}(X)$ is an inverse $\land$-semigroup.

Next, note that $E(\mathcal{I}^{\gr}(X)) = E(\mathcal{I}(X))$, and therefore satisfies condition (2) in Definition~\ref{bool-def} (see Remark~\ref{bool-rem1}), since $\mathcal{I}(X)$ is Boolean, by~\cite[Example 2.25]{Lawson2}. Hence, to conclude that $\mathcal{I}^{\gr}(X)$ is $\Gamma$-graded-Boolean, it suffices to take an orthogonal pair $\phi, \psi \in \mathcal{I}(X)_{\alpha}$ ($\alpha \in \Gamma$), and show that $\phi \lor \psi \in \mathcal{I}(X)_{\alpha}$. Being orthogonal amounts to $\dom(\phi) \cap \dom (\psi) = \emptyset$ and $\im(\phi) \cap \im (\psi) = \emptyset$, and so the join $\chi \in \mathcal{I}(X)$ of $\phi$ and $\psi$ in $\mathcal{I}(X)$ is defined by $\dom(\chi) := \dom(\phi) \cup \dom(\psi)$, $\im(\chi) := \im(\phi) \cup \im(\psi)$, $\chi(x) := \phi(x)$ for all $x \in \dom(\phi)$, and $\chi(x) := \psi(x)$ for all $x \in \dom(\psi)$. Since $\phi, \psi \in \mathcal{I}(X)_{\alpha}$, clearly also $\phi \lor \psi \in \mathcal{I}(X)_{\alpha}$.

\smallskip

(2) Suppose that the grading on $X$ is trivial. Then $\mathcal{I}^{\gr}(X) = \mathcal{I}(X)_{\varepsilon} = \mathcal{I}(X)$, and so $\mathcal{I}^{\gr}(X)$ is Boolean, by~\cite[Example 2.25]{Lawson2}. The same conclusion also follows from (1).

Conversely, suppose that $\mathcal{I}^{\gr}(X)$ is Boolean and $|X| \geq 3$. Suppose further that $X_{\alpha} \neq \emptyset$ and $X_{\beta} \neq \emptyset$ for distinct $\alpha, \beta \in \Gamma$. Let $x \in X_{\alpha}$, $y \in X_{\beta}$, and $z \in X\setminus \{x,y\}$ be arbitrary, with $z \in X_{\gamma}$ for some $\gamma \in \Gamma$. Interchanging the roles of $x$ and $y$, if necessary, we may assume that $\gamma \neq \beta$. Now let $\phi, \psi \in \mathcal{I}^{\gr}(X)$ be the unique functions such that $\dom(\phi) = \{x\}$, $\im(\phi) = \{x\}$, $\dom(\psi) = \{y\}$, and $\im(\psi) = \{z\}$. Then $\phi \in \mathcal{I}(X)_{\varepsilon}$ and $\psi \in \mathcal{I}(X)_{\gamma\beta^{-1}}$. Since $\dom(\phi) \cap \dom (\psi) = \emptyset$ and $\im(\phi) \cap \im (\psi) = \emptyset$, the elements $\phi$ and $\psi$ are orthogonal, and so must have a join $\phi \lor \psi$ in $\mathcal{I}^{\gr}(X)$, since we assumed it to be Boolean. Then, by Lemma~\ref{filter-lemma}(2),
\[\varepsilon = \deg(\phi) = \deg(\phi \lor \psi) = \deg(\psi) = \gamma\beta^{-1},\]
and so $\gamma = \beta$, contrary to hypothesis. Therefore, it cannot be the case that $X_{\alpha} \neq \emptyset$ for more than one $\alpha \in \Gamma$, and so the grading on $X$ is trivial.
\end{proof}

The previous proposition describes Boolean $\mathcal{I}^{\gr}(X)$ when $|X| \geq 3$; let us discuss the other cases. If $|X| = 1$, then only trivial gradings are possible, and so $\mathcal{I}^{\gr}(X) = \mathcal{I}(X)$ is Boolean. If $|X| = 2$, then it is possible for $\mathcal{I}^{\gr}(X)$ to be both Boolean and non-trivially graded, as we describe in the following example.

\begin{example} \label{symm-inv-eg}
Let $X := \{x,y\}$, and write 
\[\mathcal{I}(X) = \big\{0, 1, \tau, \theta_{xx}, \theta_{xy}, \theta_{yx}, \theta_{yy}\big\},\]
where $\tau$ denotes the one nontrivial permutation of $X$, and $\theta_{ij}$ is the unique element of $\mathcal{I}(X)$ having $\dom(\theta_{ij}) = \{j\}$ and $\im(\theta_{ij}) = \{i\}$ ($i,j \in X$). As mentioned above, $\mathcal{I}(X)$ is Boolean.

Now, define a grading $\phi : X \to \Z_2$, via $\phi(x) := 0$ and $\phi(y) := 1$. Then $\mathcal{I}(X)_0 = \{0, 1, \theta_{xx}, \theta_{yy}\}$ and $\mathcal{I}(X)_1 = \{0, \tau, \theta_{xy}, \theta_{yx}\}$. Hence, in this case, $\mathcal{I}(X) = \mathcal{I}^{\gr}(X)$, and the grading is non-trivial. 
\end{example}

We can build further examples of graded-Boolean inverse $\land$-semigroups using the distributive completion (see Section~\ref{basics-sect} for the notation), as well as a graded version thereof.

\begin{prop} \label{distr-comp}
Let $S$ be a $\, \Gamma$-graded inverse $\land$-semigroup. Let $\, \mathsf{FC}^{\gr}(S)$ denote the set of all homogeneous elements of $\, \mathsf{FC}(S)$, let $\, \equiv^{\gr}$ denote the restriction of $\, \equiv$ to $\, \mathsf{FC}^{\gr}(S)$, and let $\, \mathsf{D}^{\gr}(S) := \mathsf{FC}^{\gr}(S/\leftrightarrow)/\equiv^{\gr}$.

\begin{enumerate}[\upshape(1)]
\item $\mathsf{FC}^{\gr}(S)$ and $\, \mathsf{D}^{\gr}(S)$ are $\, \Gamma$-graded inverse semigroups.

\smallskip 

\item If $S$ has no zero-divisors, then $\, \mathsf{FC}(S) = \mathsf{FC}^{\gr}(S)$ and $\, \mathsf{D}(S) = \mathsf{D}^{\gr}(S)$.

\smallskip 

\item If $S$ is pre-Boolean, then $\, \mathsf{D}^{\gr}(S)$ is a $\, \Gamma$-graded-Boolean inverse $\land$-semigroup.
\end{enumerate}
\end{prop}

\begin{proof}
(1) Clearly, $\mathsf{FC}^{\gr}(S)$ is closed under multiplication and inverses, and hence is an inverse subsemigroup of $\mathsf{FC}(S)$. Also, $I \mapsto \alpha$, where $I \neq \{0\}$, and $\deg(s) = \alpha \in \Gamma$ for any nonzero $s \in I$, defines a $\Gamma$-grading on $\mathsf{FC}^{\gr}(S)$. Since $\equiv$ is a congruence on $\mathsf{FC}(S)$, it is easy to see that $\equiv^{\gr}$ is also a congruence, and so $\mathsf{FC}^{\gr}(S)/\equiv^{\gr}$ is an inverse semigroup.

Next, we show that the semigroup $S' := S/\leftrightarrow$ is $\Gamma$-graded. If $s,t \in S \setminus \{0\}$ are such that $s \to t$, then, in particular, $s \land t \neq 0$.  Hence, by Lemma~\ref{filter-lemma}(2), $\deg(s) = \deg(s\land t) = \deg(t)$. Thus if we denote by $[s]$ the congruence class of $s \in S\setminus \{0\}$ under $\leftrightarrow$, then every element of $[s]$ has the same degree. Therefore, we can turn $S'$ into a $\Gamma$-graded semigroup via $[s] \mapsto \alpha$, where $s \neq 0$, and $\deg(t) = \alpha \in \Gamma$ for any $t \in [s]$.

Since $S'$ is $\Gamma$-graded, $\mathsf{FC}^{\gr}(S')$ is $\Gamma$-graded as well (as shown in the first paragraph). Now, let $I,J \in \mathsf{FC}^{\gr}(S')$ be such that $I \equiv^{\gr} J$. Then, either $I$ and $J$ are both $\{0\}$, or both nonzero. In the latter case, by the argument in the previous paragraph, $\deg(I) = \deg(J)$. This, once again, allows us turn $\mathsf{D}^{\gr}(S) = \mathsf{FC}^{\gr}(S')/\equiv^{\gr}$ into a $\Gamma$-graded semigroup, as above.

\smallskip 

(2) Suppose that $S$ has no zero-divisors, and let $I \in \mathsf{FC}(S) \setminus \{\{0\}\}$. Then, in particular, $S$ has no nonzero orthogonal elements, and so every nonzero element of $I$ has the same degree, by Lemma~\ref{filter-lemma}(1). That is, $I$ is homogeneous, and hence $I \in \mathsf{FC}^{\gr}(S)$. Since $\mathsf{FC}^{\gr}(S)$ is a subsemigroup of $\mathsf{FC}(S)$, we conclude that $\mathsf{FC}(S) = \mathsf{FC}^{\gr}(S)$. Given that $S$ has no zero-divisors, neither does $S' := S/\leftrightarrow$, and so our argument also shows that $\mathsf{FC}(S') = \mathsf{FC}^{\gr}(S')$. It follows that $\mathsf{D}(S) = \mathsf{D}^{\gr}(S)$ as well.

\smallskip 

(3) Suppose that $S$ is pre-Boolean. Since $\mathsf{FC}^{\gr}(S/\leftrightarrow)$ is a subsemigroup of $\mathsf{FC}(S/\leftrightarrow)$, and since $\equiv^{\gr}$ is the restriction of $\equiv$ from $\mathsf{FC}(S/\leftrightarrow)$ to $\mathsf{FC}^{\gr}(S/\leftrightarrow)$, we can view $\mathsf{D}^{\gr}(S)$ as a subsemigroup of the Boolean inverse $\land$-semigroup $\mathsf{D}(S)$.

Now let $\alpha \in \Gamma$, and let $X, Y \in \mathsf{D}^{\gr}(S)_{\alpha}$. By~\cite[Lemma 2.10]{Lawson3} (see also~\cite[Lemma 2.17]{Lawson3}), the natural partial order on $\mathsf{D}(S)$ is simply set inclusion. Hence the meet of $X$ and $Y$ in $\mathsf{D}(S)$ is $X \cap Y$, while join is $X \cup Y$, in case $X$ and $Y$ are compatible. Both are elements of $\mathsf{D}^{\gr}(S)_{\alpha}$, since $X,Y \in \mathsf{D}^{\gr}(S)_{\alpha}$, and therefore, necessarily, $X \cap Y$ and $X \cup Y$ (in case $X$ and $Y$ are orthogonal) are the meet and join, respectively, of $X$ and $Y$ in $\mathsf{D}^{\gr}(S)$. Thus $\mathsf{D}^{\gr}(S)$ is an inverse $\land$-semigroup and satisfies condition (1) in Definition~\ref{bool-def}. Since $E(\mathsf{D}^{\gr}(S))$ can be viewed as a subsemigroup of $E(\mathsf{D}(S))$, considerations analogous to those above show that the properties in condition (2) of Definition~\ref{bool-def} pass from $E(\mathsf{D}(S))$ to $E(\mathsf{D}^{\gr}(S))$. (Note that, by Lemma~\ref{filter-lemma}(2), the principal order ideals of $\mathsf{D}(S)$ and $\mathsf{D}^{\gr}(S)$, and hence also of $E(\mathsf{D}(S))$ and $E(\mathsf{D}^{\gr}(S))$, coincide.) Therefore $\mathsf{D}^{\gr}(S)$ is $\Gamma$-graded-Boolean.
\end{proof}

Graph inverse semigroups, first introduced in~\cite{AH}, constitute a rich and well-studied class of inverse semigroups, constructed from directed graphs (see~\cite{MM} for an overview). According to~\cite[Theorem 4.10]{Lawson3}, graph inverse semigroups, over graphs in which each vertex has finite in-degree, are pre-Boolean. These semigroups are also $\Z$-graded, where $\Z$ denotes the group of the integers (see~\cite[Section 8.1]{RZ} for the details). Hence, by Proposition~\ref{distr-comp}(3), $\mathsf{D}^{\gr}(S)$ is a $\Z$-graded-Boolean inverse $\land$-semigroup, for any graph inverse semigroup $S$ constructed from a graph where each vertex has finite in-degree.

Generally speaking, a graph inverse semigroup has zero-divisors, and $\mathsf{D}(S)$ need not coincide with $\mathsf{D}^{\gr}(S)$, for such a semigroup $S$. However, the graph inverse semigroup corresponding to the graph with one vertex and one loop does not have zero-divisors. More specifically, this particular semigroup, known as the \emph{bicyclic semigroup}, can be presented by generators and relations as the monoid
\[P_1 := \langle x,y \mid xy = 1\rangle \cup \{0\},\]
though it is usually viewed in the literature as an inverse monoid without zero. By Proposition~\ref{distr-comp}(2,3), $\mathsf{D}(P_1) = \mathsf{D}^{\gr}(P_1)$ is a $\Z$-graded-Boolean inverse $\land$-semigroup.

We conclude this section by recording a (fairly trivial) way of creating graded-Boolean semigroups from existing ones.

\begin{prop}
Let $S$ be a $\, \Gamma$-graded inverse semigroup.
\begin{enumerate}[\upshape(1)]
\item If $S$ is an inverse $\land$-semigroup, then so is $S_{\varepsilon}$.

\smallskip 

\item If $S$ is $($$\Gamma$-graded-$)$Boolean, then $S_{\varepsilon}$ is both Boolean and $\, \Gamma$-graded-Boolean.
\end{enumerate}
\end{prop}

\begin{proof}
(1) Suppose that $S$ is an inverse $\land$-semigroup. Since, as noted in Section~\ref{gradedsemisubsection}, $S_{\varepsilon}$ is an inverse semigroup, it suffices to show that each pair of elements in $S_{\varepsilon}$ has a meet. Letting $s,t \in S_{\varepsilon}$, we have $s\land t \in S$. If $s\land t = 0$, then $s\land t \in S_{\varepsilon}$. If $s\land t \neq 0$, then $s \neq 0 \neq t$, and so Lemma~\ref{filter-lemma}(2) implies that $s\land t \in S_{\varepsilon}$. In either case $s\land t$ is necessarily the meet of $s$ and $t$ in $S_{\varepsilon}$. It follows that $S_{\varepsilon}$ is an inverse $\land$-semigroup.

\smallskip

(2) Suppose that $S$ is either Boolean or $\Gamma$-graded-Boolean. Then $S_{\varepsilon}$ has joins for orthogonal pairs, by Lemma~\ref{filter-lemma}(4). Since $E(S_{\varepsilon}) = E(S)$, we conclude, by Remark~\ref{bool-rem1}, that $S_{\varepsilon}$ is Boolean, and hence also $\Gamma$-graded-Boolean.
\end{proof}

\section{Duality} \label{duality-sect}

This section is devoted to proving our main result, that the categories $\GIS$ and $\GG$, of $\Gamma$-graded-Boolean inverse $\land$-semigroups and $\Gamma$-graded Hausdorff ample groupoids, respectively (see Definitions~\ref{is-cat-def} and~\ref{gp-cat-def} for the details), are dually equivalent. The construction is divided into a number of lemmas and propositions. We begin by defining the main ingredients.

Given an inverse semigroup $S$, let 
\[\GP(S):=\big\{X \mid X \text{~is an ultrafilter on~} (S, \leq)\big\},\] 
where $\leq$ is the natural partial order. For all $X,Y \in \GP(S)$ define 
\[X^{-1} := \big\{x^{-1} \mid x \in X\big\},\] 
\[XY :=\big \{xy \mid x \in X, y \in Y\big\},\] 
and 
\[X\cdot Y := (XY)^{\uparrow}.\] 
According to~\cite[Section 2.1]{Lawson3}, $\GP(S)$ is a groupoid with respect to the operation $\cdot$, where $(X,Y) \in \GP(S)^{(2)}$ provided that $X^{-1}\cdot X = Y\cdot Y^{-1}$, and where elements of the form $X^{-1}\cdot X$ ($X \in \GP(S))$ constitute the unit space. It follows from~\cite[Lemma 2.8(2)]{Lawson3} that an ultrafilter on $S$ belongs to the unit space of $\GP(S)$ if and only if it contains an idempotent.

Given an ample groupoid $\mathscr{G}$, let
\[\SG(\mathscr{G}) :=\big\{X\mid X \text{~is a compact slice of~} \mathscr{G}\big\}.\]
Then $\SG(\mathscr{G})$ is an inverse semigroup under the operations given in (\ref{pofgtryhf}) and (\ref{pofgtryhf2}) (see~\cite[Proposition 2.2.4]{paterson1}), with $\emptyset$ as the zero element. The idempotent set $E(\SG(\mathscr{G}))$ consists of compact open subsets of $\mathscr{G}^{(0)}$. Supposing that $\mathscr{G}$ is $\Gamma$-graded, for some group $\Gamma$, we say that a slice $X \subseteq \mathscr{G}$ is \emph{homogeneous} if $X \subseteq \mathscr{G}_\alpha$ for some $\alpha \in \Gamma$, and set
\begin{equation*}
\SG^{\gr}(\mathscr{G}) := \big\{X\mid X \text{~is a homogeneous compact slice of~} \mathscr{G}\big\}.
\end{equation*}
Since $\mathscr{G}_{\alpha^{-1}} = \mathscr{G}_{\alpha}^{-1}$ for each $\alpha \in \Gamma$, it follows that $\SG^{\gr}(\mathscr{G})$ is an inverse subsemigroup of $\SG(\mathscr{G})$. Moreover, it is easy to see that defining $\phi :\SG^{\gr}(\mathscr{G}) \backslash \{\emptyset\} \rightarrow \Gamma$ by $\phi(X) := \alpha$, whenever $X\subseteq \mathscr{G}_\alpha$, turns $\SG^{\gr}(\mathscr{G})$ into a $\Gamma$-graded inverse semigroup. The natural partial order on $\SG(\mathscr{G})$ is simply set inclusion (see \cite[Section 2.2]{paterson1}). Since the idempotent slices of $\mathscr{G}$ are subsets of $\mathscr{G}_{\varepsilon}$, and are therefore homogeneous, the natural partial order on $\SG^{\gr}(\mathscr{G})$ is also set inclusion. Finally, note that since each $\mathscr{G}_\alpha$ is clopen, and since $\mathscr{G}$ is ample, the elements of $\SG^{\gr}(\mathscr{G})$ form a base for the topology on $\mathscr{G}$.

\begin{lemma} \label{gr-inv-obj}
Let $\mathscr{G}$ be a $\, \Gamma$-graded Hausdorff ample groupoid. Then $\, \SG^{\gr}(\mathscr{G})$ is a $\, \Gamma$-graded-Boolean inverse $\land$-semigroup, with intersection as the meet operation and union as the join operation.
\end{lemma}

\begin{proof}
According to \cite[Lemma 2.32(4)]{Lawson3}, $\SG(\mathscr{G})$ is a Boolean inverse $\land$-semigroup, with intersection and union as the meet and join operations, respectively. Since the intersection of two homogeneous (compact) slices of $\mathscr{G}$ is clearly homogeneous, it follows that $\SG^{\gr}(\mathscr{G})$ is an inverse $\land$-subsemigroup of $\SG(\mathscr{G})$. Since $\SG(\mathscr{G})$ has joins of orthogonal pairs, $\SG^{\gr}(\mathscr{G})$ must have joins of homogenous orthogonal pairs, by Lemma~\ref{filter-lemma}(4). Finally, since $E(\SG(\mathscr{G})) = E(\SG^{\gr}(\mathscr{G}))$, we see, from Definition~\ref{bool-def}, that $\SG^{\gr}(\mathscr{G})$ is a $\Gamma$-graded-Boolean inverse $\land$-semigroup.
\end{proof}

\begin{lemma}[cf.\ \cite{Lawson3}, Lemma 2.36] \label{gr-inv-obj2}
Let $\mathscr{G}$ be a $\, \Gamma$-graded Hausdorff ample groupoid, and for each $y \in \mathscr{G}$ let $\mathcal{X}_y := \{Y \in \SG^{\gr}(\mathscr{G}) \mid y \in Y\}$.\begin{enumerate}[\upshape(1)]
\item $\mathcal{X}_y$ is an ultrafilter on $\, \SG^{\gr}(\mathscr{G})$, for each $y \in \mathscr{G}$.

\smallskip 

\item Every ultrafilter on $\, \SG^{\gr}(\mathscr{G})$ is of the form $\mathcal{X}_y$, for some $y \in \mathscr{G}$.

\smallskip 

\item $\mathcal{X}_y^{-1} = \mathcal{X}_{y^{-1}}$ for all $y \in \mathscr{G}$.

\smallskip 

\item For all $x,y \in \mathscr{G}$, if $(x,y) \in \mathscr{G}^{(2)}$, then $\mathcal{X}_x \cdot \mathcal{X}_y = \mathcal{X}_{xy}$.

\smallskip 

\item $\mathcal{X}_x = \mathcal{X}_y$ if and only if $x=y$, for all $x,y \in \mathscr{G}$.
\end{enumerate}
\end{lemma}

\begin{proof}
These statements can be proved in essentially the same way as their non-graded analogues in~\cite[Lemma 2.36]{Lawson3}, but we provide the arguments for the convenience of the reader.

\smallskip

(1) By Lemma~\ref{gr-inv-obj}, $\SG^{\gr}(\mathscr{G})$ is closed under finite intersections. Note also that $\mathcal{X}_y \neq \emptyset$ for each $y \in \mathscr{G}$, since the elements of $\SG^{\gr}(\mathscr{G})$ form a base for the topology on $\mathscr{G}$. Then, from the fact that the natural partial order on this inverse semigroup is set inclusion, it follows that $\mathcal{X}_y$ is a filter for each $y \in \mathscr{G}$. Now, according to \cite[Lemma 9.1.2]{Lawson1} (see, alternatively, \cite[Lemma 2.6(2)]{Lawson3}), a filter $F$ on a lower semilattice $S$ is an ultrafilter if and only if $F$ contains all $s \in S$ such that $s\land t \neq 0$ for all $t \in F$. We shall use this criterion on $\mathcal{X}_y$.

Thus suppose that $X \in \SG^{\gr}(\mathscr{G})$ is such that $X \cap Y \neq \emptyset$ for all $Y \in \mathcal{X}_y$. Note that since $X$ is compact and $\mathscr{G}$ is Hausdorff, $X$ is necessarily closed. Thus, to conclude that $X \in \mathcal{X}_y$, it suffices to show that $X \cap N \neq \emptyset$ for all open neighbourhoods $N$ of $y$. Since the elements of $\SG^{\gr}(\mathscr{G})$ form a base for the topology on $\mathscr{G}$, there exists $Y \in \SG^{\gr}(\mathscr{G})$ such that $y \in Y \subseteq N$. But then $Y \in \mathcal{X}_y$, and so $X \cap Y \neq \emptyset$, by hypothesis. Hence $X \cap N \neq \emptyset$, as desired.

\smallskip

(2) Let $\mathcal{F}$ be an ultrafilter on $\SG^{\gr}(\mathscr{G})$, let $X \in \mathcal{F}$, and let $\mathcal{F}' = \{X \cap Y \mid Y \in \mathcal{F}\}$. Since $\mathcal{F}$ is a filter, $\mathcal{F}' \subseteq \mathcal{F}$, and any finite intersection of elements of $\mathcal{F}'$ is nonempty. Moreover, each element of $\mathcal{F}'$ is closed, as a compact subset of a Hausdorff space. Thus $\mathcal{F}'$ consists of closed subsets of the compact set $X$, and has the finite intersection property. It follows that $\mathcal{F}'$ has a nonempty intersection. Taking $y$ to be any element of that intersection, we see that $\mathcal{F} \subseteq \mathcal{X}_y$. Since $\mathcal{F}$ is an ultrafilter, and $\mathcal{X}_y$ is a filter, by (1), we have $\mathcal{F} = \mathcal{X}_y$.

\smallskip

(3) Let $y \in \mathscr{G}$. According to \cite[Lemma 2.36(3)]{Lawson3},
\[\big\{Y \in \SG(\mathscr{G}) \mid y \in Y\big\}^{-1} = \big\{Y^{-1} \in \SG(\mathscr{G}) \mid y \in Y\big\} = \big\{Y \in \SG(\mathscr{G}) \mid y^{-1} \in Y\big\}.\] 
Restricting to homogeneous compact slices gives $\mathcal{X}_y^{-1} = \mathcal{X}_{y^{-1}}$.

\smallskip

(4) Let $x,y \in \mathscr{G}$ be such that $(x,y) \in \mathscr{G}^{(2)}$, and let $X \in \mathcal{X}_x$ and $Y \in \mathcal{X}_y$. Then $xy \in XY \in \SG^{\gr}(\mathscr{G})$, and so $XY \in \mathcal{X}_{xy}$. It follows that 
\[\mathcal{X}_x\mathcal{X}_y = \big\{XY \mid X \in \mathcal{X}_x, Y \in \mathcal{X}_y\big\} \subseteq \mathcal{X}_{xy},\]
and hence $\mathcal{X}_x \cdot \mathcal{X}_y = (\mathcal{X}_x \mathcal{X}_y)^{\uparrow} \subseteq \mathcal{X}_{xy}$, since $\mathcal{X}_{xy}$ is a filter, by (1). Finally, since $\mathcal{X}_x \cdot \mathcal{X}_y$ is an ultrafilter (being an element of $\GP(\SG^{\gr}(\mathscr{G})$), we conclude that $\mathcal{X}_x \cdot \mathcal{X}_y = \mathcal{X}_{xy}$.

\smallskip

(5) Certainly, if $x=y$, then $\mathcal{X}_x = \mathcal{X}_y$. For the converse, suppose that $x,y \in \mathscr{G}$ are distinct. Since $\mathscr{G}$ is Hausdorff, and since the elements of $\SG^{\gr}(\mathscr{G})$ form a base for the topology on $\mathscr{G}$, there exist disjoint $X, Y \in \SG^{\gr}(\mathscr{G})$ such that $x\in X$ and $y \in Y$. Then $X \in \mathcal{X}_x \setminus \mathcal{X}_y$, and so $\mathcal{X}_x \neq \mathcal{X}_y$.
\end{proof}

\begin{lemma} \label{gr-inv-morph}
Let $\phi: \mathscr{G} \to \mathscr{H}$ be a morphism in $\, \GG$. Then $\phi^{-1} : \SG^{\gr}(\mathscr{H}) \to \SG^{\gr}(\mathscr{G})$ is a morphism in $\, \GIS$. Moreover, if $\phi$ is a functor $\, ($rather than a partial functor$)$, then $\phi^{-1}$ is a proper homomorphism.
\end{lemma}

\begin{proof}
Let $\mathscr{G}'$ denote the domain of $\phi$, and let $X \in \SG^{\gr}(\mathscr{H})$ be arbitrary. Since $\phi$ is continuous, $\phi^{-1}(X)$ is an open subset of $\mathscr{G}'$, and hence also of $\mathscr{G}$. Since $\phi$ is proper, $\phi^{-1}(X)$ is a compact subset of ($\mathscr{G}'$ and) $\mathscr{G}$. Since $\phi$ is star-injective, we therefore conclude that $\phi^{-1}(X) \in \SG(\mathscr{G})$ (see \cite[Proposition 2.17]{Lawson2} for more details), and so $\phi^{-1} : \SG(\mathscr{H}) \to \SG(\mathscr{G})$ is a well-defined function. Supposing that $X \subseteq \mathscr{H}_{\alpha}$ for some $\alpha \in \Gamma$, we must have $\phi^{-1}(X) \subseteq \mathscr{G}_{\alpha}$, since $\phi$ is graded. Hence $\phi^{-1}$ restricts to a graded function $\phi^{-1} : \SG^{\gr}(\mathscr{H}) \to \SG^{\gr}(\mathscr{G})$.

To prove that $\phi^{-1}$ is a homomorphism, let $X,Y \in \SG^{\gr}(\mathscr{H})$. Suppose that $z \in \phi^{-1}(XY)$, and write $\phi(z)=xy$ for some $x \in X$ and $y \in Y$. Since $\phi: \mathscr{G}' \to \mathscr{H}$ is a covering functor, by \cite[Lemma 2.16]{Lawson2}, there exist $u,v \in \mathscr{G}'$ such that $z = uv$, $\phi(u) = x$, and $\phi(v) = y$. Thus $u \in \phi^{-1}(X)$ and $v \in \phi^{-1}(Y)$, from which we conclude that $z \in \phi^{-1}(X)\phi^{-1}(Y)$. Hence $\phi^{-1}(XY) \subseteq \phi^{-1}(X)\phi^{-1}(Y)$. For the opposite inclusion, suppose that $z \in \phi^{-1}(X)\phi^{-1}(Y)$, and write $z = uv$, where $\phi(u) \in X$ and $\phi(v) \in Y$. Since $\phi: \mathscr{G}' \to \mathscr{H}$ is a functor, we have $\phi(z) =\phi(u)\phi(v) \in XY$, and so $z \in  \phi^{-1}(XY)$. Therefore $\phi^{-1}(XY) = \phi^{-1}(X)\phi^{-1}(Y)$.

Since the natural partial order on each of $\SG^{\gr}(\mathscr{H})$ and $\SG^{\gr}(\mathscr{G})$ is set inclusion, $\phi^{-1}$ preserves meets. Also, clearly, $\phi^{-1} : \SG^{\gr}(\mathscr{H}) \to \SG^{\gr}(\mathscr{G})$ restricts to a function $E(\SG^{\gr}(\mathscr{H})) \to E(\SG^{\gr}(\mathscr{G}))$. Since the natural partial order on each of  $E(\SG^{\gr}(\mathscr{H}))$ and $E(\SG^{\gr}(\mathscr{G}))$ is likewise set inclusion, this function $E(\SG^{\gr}(\mathscr{H})) \to E(\SG^{\gr}(\mathscr{G}))$ is a lattice homomorphism   that preserves complements in principal order ideals. Thus $\phi^{-1}$ is a morphism in $\GIS$.

Finally, let us suppose that $\phi$ is a functor, and show that $\phi^{-1}$ is proper. To that end, let $\mathcal{F}$ be an ultrafilter on $\SG^{\gr}(\mathscr{G})$. Then, by Lemma~\ref{gr-inv-obj2}(1,2), $\mathcal{F} = \mathcal{X}_y$, for some $y \in \mathscr{G}$. Likewise, given that $\phi(y)$ is defined, $\mathcal{X}_{\phi(y)}$ is an ultrafilter on $\SG^{\gr}(\mathscr{H})$. The inverse image of $\mathcal{F}$ under $\phi^{-1}$ consists of all $X \in \SG^{\gr}(\mathscr{H})$ such that $\phi^{-1}(X) \in \mathcal{F}$. But $\phi^{-1}(X) \in \mathcal{F}$ if and only if $y \in \phi^{-1}(X)$ if and only if $\phi(y) \in X$ if and only if $X \in \mathcal{X}_{\phi(y)}$. Thus the inverse image of $\mathcal{F} = \mathcal{X}_y$ under $\phi^{-1}$ is the ultrafilter $\mathcal{X}_{\phi(y)}$, from which it follows that $\phi^{-1}$ is proper.
\end{proof}

\begin{prop} \label{gp-to-sg-fuct}
Sending each object $\mathscr{G}$ to $\, \SG^{\gr}(\mathscr{G})$, and each morphism $\phi$ to $\phi^{-1}$, defines a contravariant functor $\, \SG^{\gr} : \GG \to \GIS$.
\end{prop}

\begin{proof}
The mapping in question clearly respects identity morphisms and composition of morphisms. Thus the desired conclusion follows from Lemmas~\ref{gr-inv-obj} and~\ref{gr-inv-morph}.
\end{proof}

Now we turn to building a functor in the opposite direction. The next lemma provides the main slice construction for our groupoids, and records some basic observations that can be carried over from the non-graded context unaltered. 

\begin{lemma}[\cite{Lawson3}, Lemmas 2.10, 2.15(1)] \label{inv-slice}
Let $S$ be an inverse $\land$-semigroup, and for each $s \in S$ let $\mathcal{Y}_s := \{Y \in \GP(S) \mid s \in Y\}$.
\begin{enumerate}[\upshape(1)]
\item $\mathcal{Y}_s \cap \mathcal{Y}_t = \mathcal{Y}_{s\land t}$ for all $s,t \in S$.

\smallskip 

\item The restrictions of $\domr$ and $\ran$ to $\mathcal{Y}_s$ are injective, for each $s \in S$.

\smallskip 

\item $\mathcal{Y}_s^{-1} = \mathcal{Y}_{s^{-1}}$ for each $s \in S$.

\smallskip 

\item $\mathcal{Y}_s\mathcal{Y}_t = \mathcal{Y}_{st}$ for all $s,t \in S$.

\smallskip 

\item If $S$ is separative, then for all $s_1, \dots, s_n \in S$, the restrictions of $\domr$ and $\ran$ to $\, \bigcup_{i=1}^n \mathcal{Y}_{s_i}$ are injective if and only if $\, \{s_1, \dots, s_n\}$ is compatible.
\end{enumerate}
\end{lemma}

\begin{lemma} \label{gr-gp-obj}
Let $S$ be a nonzero $\, \Gamma$-graded inverse $\land$-semigroup, graded via $\phi : S\setminus \{0\} \to \Gamma$. For each $X \in \GP(S)$, define $\psi: \GP(S) \to \Gamma$ by $\psi(X) := \phi(x)$, where $x \in X$ is arbitrary. 
\begin{enumerate}[\upshape(1)]
\item $\GP(S)$ is a Hausdorff \'etale groupoid in the topology generated by basic open sets of the form $\mathcal{Y}_s$, as in Lemma~\ref{inv-slice}. 

\smallskip

\item $\GP(S)$ is a $\, \Gamma$-graded groupoid, with respect to $\psi$.

\smallskip

\item $\GP(S)_{\varepsilon} = \GP(S_{\varepsilon})$.

\smallskip

\item If $S$ is graded-Boolean, then $\, \GP(S)$ is ample.
\end{enumerate}
\end{lemma}

\begin{proof}
(1) According to \cite[Proposition 2.33]{Lawson3}, for each inverse $\land$-semigroup $S$, $\GP(S)$ is a Hausdorff \'etale groupoid in the topology described in the statement. 

\smallskip

(2) We first note that $\psi$ is well-defined, by Lemma~\ref{filter-lemma}(3). Now, let $(X,Y) \in \GP(S)^{(2)}$, and let $x \in X$ and $y \in Y$ be arbitrary. Then $xy \in X \cdot Y = (XY)^{\uparrow}$, and $xy \neq 0$, since $X \cdot Y \in \GP(S)$ is an ultrafilter. Thus, 
\[\psi(X\cdot Y) = \phi(xy) = \phi(x)\phi(y) = \psi(X)\psi(Y).\]
Finally, $\psi^{-1}(\alpha) = \bigcup_{x \in \phi^{-1}(\alpha)}\mathcal{Y}_x$ is open for all $\alpha \in \Gamma$, and so $\psi$ is continuous. Therefore $\psi: \GP(S) \to \Gamma$ is a grading.

\smallskip

(3) Since $\phi(x) = \varepsilon$ for all $x \in X$ such that $X \in \GP(S)_{\varepsilon}$, we have $X \subseteq S_{\varepsilon}$, from which it follows that $X$ is an ultrafilter on $S_{\varepsilon}$, and so $\GP(S)_{\varepsilon} \subseteq \GP(S_{\varepsilon})$. For the opposite inclusion, suppose that $X \in \GP(S_{\varepsilon})$. Since, by Lemma~\ref{filter-lemma}(3), any filter on $S$ containing $X$ must consist of elements of $S_{\varepsilon}$, we conclude that $X$ is an ultrafilter on $S$. Hence $\GP(S)_{\varepsilon} \supseteq \GP(S_{\varepsilon})$, and so $\GP(S)_{\varepsilon} = \GP(S_{\varepsilon})$.

\smallskip

(4)  According to~\cite[Theorem 2.34]{Lawson3}, $\GP(S)$ a Hausdorff ample groupoid for any pre-Boolean $\land$-semigroup $S$ (with respect to the topology described in (1)). The desired conclusion now follows from Proposition~\ref{pre-bool}.
\end{proof}

\begin{lemma}[cf.\ \cite{Lawson2}, Lemma 2.21(5,7,9-12)] \label{inv-slice2}
Let $S$ be a $\, \Gamma$-graded-Boolean inverse $\land$-semigroup, and view $\, \GP(S)$ as a topological groupoid in the topology generated by basic open sets of the form $\mathcal{Y}_s$, as in Lemma~\ref{inv-slice}.
\begin{enumerate}[\upshape(1)]
\item $\mathcal{Y}_s \subseteq \mathcal{Y}_t$ if and only if $s \leq t$, for all $s,t \in S$.

\smallskip 

\item If $s\lor t$ exists, then $\mathcal{Y}_s \cup \mathcal{Y}_t = \mathcal{Y}_{s\lor t}$, for all $s,t \in S$.

\smallskip

\item Let $s, t_i \in S$ be such that each $t_i \leq s$, with $i \in I$, for some indexing set $I$. Then $\mathcal{Y}_{s} = \bigcup_{i\in I} \mathcal{Y}_{t_i}$ if and only if $\mathcal{Y}_{s^{-1}s} = \bigcup_{i\in I} \mathcal{Y}_{s^{-1}t_i}$.

\smallskip

\item Let $u, v_i \in E(S)$, with $i \in I$, for some indexing set $I$. If $\mathcal{Y}_{u} = \bigcup_{i\in I} \mathcal{Y}_{v_i}$, then $\mathcal{Y}_{u} = \bigcup_{i\in J} \mathcal{Y}_{v_i}$ for some finite $J \subseteq I$. 

\smallskip

\item $\mathcal{Y}_s$ is compact for each $s \in S$.

\smallskip 

\item Each nonempty homogeneous compact slice of $\, \GP(S)$ is of the form $\, \bigcup_{i=1}^n \mathcal{Y}_{s_i}$, for some $\alpha \in \Gamma$ and compatible $\, \{s_1, \dots, s_n\} \subseteq S_{\alpha} \setminus \{0\}$.
\end{enumerate}
\end{lemma}

\begin{proof}
The proofs of most of these statements are nearly identical to their (non-graded) monoid analogues in~\cite[Lemma 2.21]{Lawson2}, but we include them for the convenience of the reader.

\smallskip

(1) Since $S$ is separative (see Section~\ref{basics-sect} for the definition), by Proposition~\ref{pre-bool}, this follows from~ \cite[Lemma 2.11(1)]{Lawson3}, which says that  $\mathcal{Y}_s \subseteq \mathcal{Y}_t$ if and only if $s \to t$, for all $s,t \in S$.

\smallskip

(2) Let $s,t \in S$, and suppose that $s \lor t$ exists. Then $\mathcal{Y}_s \cup \mathcal{Y}_t \subseteq \mathcal{Y}_{s\lor t}$, by (1). Also, by Lemma~\ref{ultrafilt-lem}, each $X \in \mathcal{Y}_{s\lor t}$ contains either $s$ or $t$, and so $\mathcal{Y}_{s\lor t} \subseteq \mathcal{Y}_s \cup \mathcal{Y}_t$.

\smallskip

(3) Suppose that $\mathcal{Y}_{s^{-1}s} = \bigcup_{i\in I} \mathcal{Y}_{s^{-1}t_i}$. Since $t_i \leq s$ for any $i \in I$, if $X \in \mathcal{Y}_{t_i}$, then $s \in X$, and so $X \in \mathcal{Y}_{s}$, showing that $\bigcup_{i\in I} \mathcal{Y}_{t_i} \subseteq \mathcal{Y}_{s}$. For the opposite inclusion, suppose that $X \in \mathcal{Y}_{s}$. Then $s \in X$, and so $s^{-1}s \in (X^{-1}X)^{\uparrow}$, which is an ultrafilter, by Lemma~\ref{idempt-filt-lem}(5). Thus $(X^{-1}X)^{\uparrow} \in \mathcal{Y}_{s^{-1}s}$, and so, by hypothesis, $(X^{-1}X)^{\uparrow} \in \mathcal{Y}_{s^{-1}t_i}$ for some $i \in I$. It follows that $s^{-1}t_i \in (X^{-1}X)^{\uparrow}$. Now, $(s(X^{-1}X)^{\uparrow})^{\uparrow} = X$, by Lemma~\ref{idempt-filt-lem}(3), and hence $ss^{-1}t_i \in X$. But $ss^{-1}t_i \leq t_i$, and so $t_i \in X$, showing that $X \in \mathcal{Y}_{t_i}$. Therefore $\mathcal{Y}_{s} = \bigcup_{i\in I} \mathcal{Y}_{t_i}$.

Conversely, suppose that $\mathcal{Y}_{s} = \bigcup_{i\in I} \mathcal{Y}_{t_i}$. Let $X \in \mathcal{Y}_{s^{-1}t_i}$ for some $i \in I$. Then $s^{-1}t_i \in X$. Since $t_i \leq s$, we have $s^{-1}t_i \leq s^{-1}s$, and hence $s^{-1}s \in X$. Therefore $X \in \mathcal{Y}_{s^{-1}s}$, and so $\bigcup_{i\in I} \mathcal{Y}_{s^{-1}t_i} \subseteq \mathcal{Y}_{s^{-1}s}$. For the opposite inclusion, let $X \in \mathcal{Y}_{s^{-1}s}$. Then $s^{-1}s \in X$. It follows, by Lemma~\ref{idempt-filt-lem}(4), that $(sX)^{\uparrow}$ is an ultrafilter, containing $s$, and so $(sX)^{\uparrow}\in \mathcal{Y}_{s} = \bigcup_{i\in I} \mathcal{Y}_{t_i}$. Hence $(sX)^{\uparrow}\in \mathcal{Y}_{t_i}$ for some $i \in I$, and therefore $sp \leq t_i$ for some $p\in X$. Then $s^{-1}sp \leq s^{-1}t_i$. Since $s^{-1}s, p \in X$, and $X$ is a semigroup, by~\cite[Lemma 2.8(2)]{Lawson3}, we have $s^{-1}sp \in X$, and so $s^{-1}t_i \in X$. Thus $X \in \mathcal{Y}_{s^{-1}t_i}$, and so $\mathcal{Y}_{s^{-1}s} = \bigcup_{i\in I} \mathcal{Y}_{s^{-1}t_i}$.

\smallskip

(4) Suppose that $\mathcal{Y}_{u} = \bigcup_{i\in I} \mathcal{Y}_{v_i}$. Then $v_i \leq u$ for each $i \in I$, by (1). Since each principal order ideal of $E(S)$ is a Boolean algebra, for each $i \in I$ there exists $v_i' \in \{u\}^{\downarrow}$ such that $v_i \land v_i' = 0$ and $v_i \lor v_i' = u$. Note that, by Lemma~\ref{comp-lem}(1), $v_i \lor v_i' = u$ (and also $v_i \land v_i' = 0$), as elements of $S$. We claim that $\mathcal{Y}_{u} \setminus \mathcal{Y}_{v_i} = \mathcal{Y}_{v_i'}$. Any ultrafilter containing ${v_i'}$ must contain $u$, since $v_i' \leq u$, but not $v_i$, since otherwise it would contain $0$. Thus $\mathcal{Y}_{v_i'} \subseteq \mathcal{Y}_{u} \setminus \mathcal{Y}_{v_i}$. For the opposite inclusion, suppose that $X \in \mathcal{Y}_{u} \setminus \mathcal{Y}_{v_i}$. Then $v_i \lor v_i' \in X$ but $v_i \notin X$, and so $v_i' \in X$, by Lemma~\ref{ultrafilt-lem}. Thus $X \in \mathcal{Y}_{v_i'}$, and so $\mathcal{Y}_{u} \setminus \mathcal{Y}_{v_i} = \mathcal{Y}_{v_i'}$.

Now, seeking a contradiction, suppose that for every finite $J \subseteq I$ we have $\mathcal{Y}_{u} \neq \bigcup_{i\in J} \mathcal{Y}_{v_i}$. Let $J \subseteq I$ be finite and nonempty. Then 
\[\emptyset \neq \mathcal{Y}_{u} \setminus \bigg(\bigcup_{i\in J}\mathcal{Y}_{v_i}\bigg) = \bigcap_{i\in J} (\mathcal{Y}_{u} \setminus \mathcal{Y}_{v_i}) = \bigcap_{i\in J} \mathcal{Y}_{v_i'} = \mathcal{Y}_{w},\]
where $w = \bigwedge_{i\in J} v_i'$, by Lemma~\ref{inv-slice}(1), and so $\bigwedge_{i\in J} v_i' \neq 0$. Thus every finite nonempty subset of $\{v_i' \mid i \in I\}$ has a nonzero meet. It follows that $\{v_i' \mid i \in I\}$ is contained in an ultrafilter $Z$ on $S$. As in the previous paragraph, we must have $u \in Z$, but $v_i \notin Z$ for all $i \in I$. Thus $Z \in \mathcal{Y}_{u}$ but $Z \notin \bigcup_{i\in I} \mathcal{Y}_{v_i}$, contrary to hypothesis. Therefore $\mathcal{Y}_{u} = \bigcup_{i\in J} \mathcal{Y}_{v_i}$ for some finite $J \subseteq I$. 

\smallskip

(5) Suppose that $\mathcal{Y}_{s} \subseteq \bigcup_{i\in I} \mathcal{Y}_{t_i}$, for some $s, t_i \in S$ and indexing set $I$. Then 
\[\mathcal{Y}_{s} = \bigcup_{i\in I} (\mathcal{Y}_{s} \cap \mathcal{Y}_{t_i}) = \bigcup_{i\in I} \mathcal{Y}_{s \land t_i},\] 
by Lemma~\ref{inv-slice}(1). Hence, $\mathcal{Y}_{s^{-1}s} = \bigcup_{i\in I} \mathcal{Y}_{s^{-1}(s \land t_i)}$, by (3). Since $s \land t_i \leq s$, it follows that each $s^{-1}(s \land t_i)$ is an idempotent. Thus, by (4), $\mathcal{Y}_{s^{-1}s} = \bigcup_{i\in J} \mathcal{Y}_{s^{-1}(s \land t_i)}$ for some finite $J \subseteq I$. Hence $\mathcal{Y}_{s} = \bigcup_{i\in J} \mathcal{Y}_{s \land t_i}$, by (3). Using Lemma~\ref{inv-slice}(1) once more, we have $\mathcal{Y}_{s} = \bigcup_{i\in J} (\mathcal{Y}_{s} \cap \mathcal{Y}_{t_i})$, and so $\mathcal{Y}_{s} \subseteq \bigcup_{i\in J} \mathcal{Y}_{t_i}$. It follows that $\mathcal{Y}_{s}$ is compact for each $s \in S$.

\smallskip

(6) Let $X$ be a nonempty homogeneous compact slice of $\GP(S)$, where the grading on $\GP(S)$ is as in Lemma~\ref{gr-gp-obj}. Since $X$ is an open and compact subset of $\GP(S)$, we have $X = \bigcup_{i=1}^n \mathcal{Y}_{s_i}$ for some $s_1, \dots, s_n \in S$. Since $X \neq \emptyset$, we may assume that each of the $s_i$ is nonzero. Since $X$ is a slice, $\{s_1, \dots, s_n\}$ is compatible, by Lemma~\ref{inv-slice}(5). (Recall that $S$ is separative, by Proposition~\ref{pre-bool}.) Now, for each $i \in I$ and $Y \in \mathcal{Y}_{s_i}$, we have $\deg(Y)=\deg(s_i)$, by Lemma~\ref{gr-gp-obj}. Since $X$ is homogenous, it follows that $\{s_1, \dots, s_n\} \subseteq S_{\alpha} \setminus \{0\}$ for some $\alpha \in \Gamma$.
\end{proof}

In preparation for Lemma~\ref{gr-gp-morph} we record the following basic fact.

\begin{lemma} \label{joins-rem}
Let $\phi : S \to T$ be a morphism in $\, \GIS$, and let $s,t \in S$ be elements that have a join $s \lor t \in S$. Then $\phi(s \lor t) = \phi(s) \lor \phi(t)$.
\end{lemma}

\begin{proof}
Setting $r: = s \lor t$, we have $r^{-1}r = s^{-1}s \lor t^{-1}t$, by \cite[Lemma 2.4.4(1)]{Lawson1}. Since $\phi$ is a semigroup homomorphism that restricts to a lattice homomorphism $E(S) \to E(T)$,
\[\phi(s \lor t) = \phi(rr^{-1}r) = \phi(r)\phi(r^{-1}r) = \phi(r)\phi(s^{-1}s \lor t^{-1}t) = \phi(r)(\phi(s^{-1}s) \lor \phi(t^{-1}t)).\]
Note that, by Lemma~\ref{comp-lem}(1), $\phi(s^{-1}s) \lor \phi(t^{-1}t)$ is the join of $\phi(s^{-1}s)$ and $\phi(t^{-1}t)$, as elements of $T$. By Remark~\ref{bool-rem2},
\[\phi(r)(\phi(s^{-1}s) \lor \phi(t^{-1}t)) = \phi(rs^{-1}s) \lor \phi(rt^{-1}t).\]
Finally, writing $s = ru$ and $t = rv$ for some $u,v \in E(S)$, we conclude that
\[\phi(s \lor t) = \phi(rur^{-1}ru) \lor \phi(rvr^{-1}rv) = \phi(rr^{-1}ru) \lor \phi(rr^{-1}rv) = \phi(s) \lor \phi(t),\]
as claimed.
\end{proof}

\begin{lemma} \label{gr-gp-morph}
Let $\phi : S \to T$ be a morphism in $\, \GIS$,
\[\mathcal{D} := \big\{X \in \GP(T) \mid \phi^{-1}(X) \in \GP(S)\big\}, \text{ and } \mathcal{D}' := \big\{X \in \GP(T) \mid \phi^{-1}(X) \neq \emptyset\big\}.\] 
Then $\mathcal{D} = \mathcal{D}'$, $\phi^{-1}$ induces a partial functor $\phi': \GP(T) \to \GP(S)$ with domain $\mathcal{D}$, and $\phi'$ is a morphism in $\, \GG$. Moreover, if $\phi : S \to T$ is a proper morphism, then $\phi' = \phi^{-1}$ is a functor.
\end{lemma}

\begin{proof}
Clearly, $\mathcal{D} \subseteq \mathcal{D}'$. To show the opposite inclusion, let $X \in \mathcal{D}'$. Since $0 \notin X$, also $0 \notin \phi^{-1}(X)$ (the hypotheses in Definition~\ref{is-cat-def} imply that $\phi(0) = 0$). Next, let $s \in \phi^{-1}(X)$ and $t \in S$ be such that $s \leq t$, and write $s = tu$ for some $u \in E(S)$. Then $\phi(s) = \phi(t)\phi(u)$, where $\phi(u) \in E(T)$, and so $\phi(s) \leq \phi(t)$. Since $\phi(s) \in X$, and $X$ is a filter, we have $\phi(t) \in X$, and so $t \in \phi^{-1}(X)$. Thus $\phi^{-1}(X) = \phi^{-1}(X)^{\uparrow}$. Now suppose that $s,t \in \phi^{-1}(X)$. Since $S$ is an inverse $\land$-semigroup, $s \land t \in S$. Since $\phi$ preserves meets, $\phi(s \land t) = \phi(s) \land \phi(t)$. Since $X$ is downward directed, we can find $r \in X$ such that $r \leq \phi(s)$ and $r \leq \phi(t)$, and so $r \leq \phi(s) \land \phi(t)$. Since $X = X^{\uparrow}$, we conclude that $\phi(s \land t) = \phi(s) \land \phi(t) \in X$, and hence $s \land t \in \phi^{-1}(X)$. Given that $\phi^{-1}(X)$ is nonempty, it is therefore downward directed, and thus a filter. Finally, suppose that $s,t \in S$ are elements that have a join $s \lor t$ in $S$, and $s \lor t \in \phi^{-1}(X)$. By Lemma~\ref{joins-rem}, $\phi(s) \lor \phi(t) = \phi(s \lor t) \in X$. Since $X$ is an ultrafilter, $\phi(s) \in X$ or  $\phi(t) \in X$, by Lemma~\ref{ultrafilt-lem}. Thus $s \in \phi^{-1}(X)$ or $t \in \phi^{-1}(X)$, and so $\phi^{-1}(X)$ is a prime filter. Applying Lemma~\ref{ultrafilt-lem} once again, we conclude that $\phi^{-1}(X) \in \GP(S)$, and so $X \in \mathcal{D}$. Thus $\mathcal{D} = \mathcal{D}'$.

Next, we show that $\mathcal{D}$ is an open subgroupoid of $\GP(T)$. Suppose that $X \in \mathcal{D}$, take $s \in \phi^{-1}(X)$, and let $Y \in \GP(T)$ be such that $\phi(s) \in Y$. Then $\phi^{-1}(Y) \neq \emptyset$, and so $Y \in \mathcal{D}' = \mathcal{D}$. Hence 
\[X \in \mathcal{Y}_{\phi(s)} = \big\{Y \in \GP(T) \mid \phi(s) \in Y\big\} \subseteq \mathcal{D},\]
and so we conclude that $\mathcal{D}$ is open (see Lemma~\ref{gr-gp-obj}(1)). Now, suppose that $X_1,X_2 \in \mathcal{D}$ are such that $(X_1,X_2) \in \GP(T)^{(2)}$, and let $s_1 \in \phi^{-1}(X_1)$ and $s_2 \in \phi^{-1}(X_2)$. Then $\phi(s_1s_2) \in X_1 \cdot X_2$, and so $\phi^{-1}(X_1 \cdot X_2) \neq \emptyset$. Also $\phi(s_1^{-1}) = \phi(s_1)^{-1} \in X_1^{-1}$, implies that $\phi^{-1}(X_1^{-1}) \neq \emptyset$. Thus $X_1 \cdot X_2, X_1^{-1} \in \mathcal{D}$, from which it follows that $\mathcal{D}$ is a subgroupoid of $\GP(T)$.

Thus $\phi^{-1}$ induces a partial function $\phi': \GP(T) \to \GP(S)$ with domain $\mathcal{D}$--an open subgroupoid of $\GP(T)$. If $\phi : S \to T$ is a proper morphism, then $\phi^{-1}$ takes all ultrafilters in $\GP(T)$ to ultrafilters in $\GP(S)$, and so $\phi' = \phi^{-1}$ is defined on all of $\GP(T)$. Hence the second claim in the statement follows from the first.

Next, let us show that $\phi'$ is graded. We may assume that $\mathcal{D} \neq \emptyset$, since otherwise that holds vacuously. Let $X \in \mathcal{D}$, let $\alpha := \deg(X)$ (see Lemma~\ref{gr-gp-obj}), and let $s \in \phi^{-1}(X)$ (which is necessarily nonzero). Since $\phi$ is graded, and $\phi(s) \in X$, it must be the case that $\deg(s) = \alpha$. Thus $\deg(\phi^{-1}(X)) = \alpha$, by Lemma~\ref{gr-gp-obj}, and so $\phi'$ is graded.

To show that $\phi'$ is continuous, let $\mathcal{Y}_s$ ($s \in S$) be a basic open set in $\GP(S)$ (see Lemma~\ref{gr-gp-obj}) such that $(\phi')^{-1}(\mathcal{Y}_s) \neq \emptyset$, and let $X \in \GP(T)$. Then $X \in (\phi')^{-1}(\mathcal{Y}_s)$ if and only if $\phi^{-1}(X) \in \mathcal{Y}_s$ if and only if $s \in \phi^{-1}(X)$ if and only if $\phi(s) \in X$ if and only if $X \in \mathcal{Y}_{\phi(s)}$. Thus $(\phi')^{-1}(\mathcal{Y}_s) = \mathcal{Y}_{\phi(s)}$, and is hence open. It follows that $\phi'$ is continuous.

To show that $\phi'$ is proper, let us take a compact $Z \subseteq \GP(S)$, and show that $(\phi')^{-1}(Z)\subseteq \mathcal{D}$ is compact. We have $Z \subseteq \bigcup_{s\in S} \mathcal{Y}_s$, and so $Z \subseteq \bigcup_{i=1}^n \mathcal{Y}_{s_i}$ for some (finite list) $s_1, \dots, s_n \in S$, since $Z$ is compact. As shown in the previous paragraph, $(\phi')^{-1}(\mathcal{Y}_{s_i})$ is either empty or $\mathcal{Y}_{\phi(s_i)}$, for each $i$, and so $(\phi')^{-1}(Z) \subseteq \bigcup_{i=1}^n \mathcal{Y}_{\phi(s_i)}$. As a compact subset of the Hausdorff space $\GP(S)$ (see Lemma~\ref{gr-gp-obj}), $Z$ is closed, and so $(\phi')^{-1}(Z)$ is closed as well, since $\phi'$ is continuous. Also, as a finite union of compact subsets of $\GP(T)$, and hence also $\mathcal{D}$ (see Lemma~\ref{inv-slice2}(5)), $\bigcup_{i=1}^n \mathcal{Y}_{\phi(s_i)}$ is compact. Therefore, being a closed subset of a compact set, $(\phi')^{-1}(Z)$ is compact.

To conclude that $\phi'$ is a morphism in $\GG$ it now suffices to check that $\phi^{-1} : \mathcal{D} \to \GP(S)$ is a covering functor. The analogue of this statement for proper homomorphisms of Boolean inverse monoids is shown in \cite[Proposition 2.15]{Lawson2}, but the (somewhat lengthy) proof, word-for-word, applies to our situation as well, modulo adjusting notation.
\end{proof}

\begin{prop} \label{smgp-to-gp-top}
Sending each object $S$ to $\, \GP(S)$, and each morphism $\phi$ to $\phi'$, as in Lemma~\ref{gr-gp-morph}, defines a contravariant functor $\, \GP : \GIS \to \GG$. Here each $\, \GP(S)$ is viewed as a topological groupoid in the topology generated by basic open sets of the form $\mathcal{Y}_s$, as in Lemma~\ref{inv-slice}.
\end{prop}

\begin{proof}
Let $\phi : S_1 \to S_2$ and $\psi : S_2 \to S_3$ be morphisms in $\GIS$. In the notation of Lemma~\ref{gr-gp-morph}, the domain of $\phi' \psi'$ is
\[(\psi')^{-1}\big(\big\{Y \in \GP(S_2) \mid \phi^{-1}(Y) \neq \emptyset\big\} \cap \big\{\psi^{-1}(X) \mid X \in \GP(S_3), \psi^{-1}(X) \neq \emptyset\big\}\big)\] 
\[= \big\{X \in \GP(S_3) \mid \phi^{-1} \psi^{-1}(X) \neq \emptyset\big\} = \big\{X \in \GP(S_3) \mid (\psi \phi)^{-1}(X) \neq \emptyset\big\},\]
which is the domain of $(\psi \phi)'$. It follows that $\GP(\psi \phi) = \GP(\phi) \GP(\psi)$. Since $\GP$ clearly respects identity morphisms, the desired conclusion follows from Lemmas~\ref{gr-gp-obj}(1,2,4) and~\ref{gr-gp-morph}.
\end{proof}

To show that the functors we have constructed, between the categories $\GIS$ and $\GG$, are dually equivalent, it remains to supply the relevant natural isomorphisms. A version of the claim, in the next lemma, that $\mathscr{G} \cong  \GP(\SG^{\gr}(\mathscr{G}))$ as graded groupoids, where $\mathscr{G}$ is not necessarily Hausdorff, can be found in \cite[Theorem 3.3]{steinberg}.

\begin{lemma} \label{gp-nat-iso}
Defining a function   
\begin{align}
\mu_{\mathscr{G}} : \mathscr{G} & \longrightarrow \GP(\SG^{\gr}(\mathscr{G})) \notag \\
x &\longmapsto \mathcal{X}_x  \notag
\end{align}
for each object $\mathscr{G}$ in $\, \GG$, gives a natural isomorphism $\mu :1_{\GG} \rightarrow \GP\SG^{\gr}$, where $\, 1_{\GG}$ denotes the identity functor on $\, \GG$.
\end{lemma}

\begin{proof}
Lemma~\ref{gr-inv-obj2} implies that each $\mu_{\mathscr{G}}$ is a bijective covering functor. Also, for each $x \in \mathscr{G}$ and each $X \in \SG^{\gr}(\mathscr{G})$ such that $x \in X$, we have $\deg(X) = \deg(x)$, since $X$ is homogeneous. Then, by Lemma~\ref{gr-gp-obj}, $\deg(\mathcal{X}_x) = \deg(x)$, and so $\mu_{\mathscr{G}}$ is graded. Thus, to conclude that $\mu_{\mathscr{G}}$ is an isomorphism in $\GG$ it suffices to prove that it is open and continuous. 

To show that $\mu_{\mathscr{G}}$ is open, let $X \in \SG^{\gr}(\mathscr{G})$ be a homogeneous compact slice of $\mathscr{G}$ (i.e., a basic open set). Then $\mu_{\mathscr{G}}(X) = \{\mathcal{X}_x \mid x \in X\}$. Now, $X \in \mathcal{X}_x$ for an ultrafilter $\mathcal{X}_x$ on $\SG^{\gr}(\mathscr{G})$ (see Lemma~\ref{gr-inv-obj2}), if and only if $x \in X$. So $\mu_{\mathscr{G}}(X)$ is precisely the set of ultrafilters on $\SG^{\gr}(\mathscr{G})$ that contain $X$. That is, $\mu_{\mathscr{G}}(X) = \mathcal{Y}_X$ (see Lemma~\ref{inv-slice}), and is therefore open, by Proposition~\ref{smgp-to-gp-top}. It follows that $\mu_{\mathscr{G}}$ is an open map.

To show that $\mu_{\mathscr{G}}$ is continuous, let $\mathcal{Y}_X$ be a basic open subset of $\GP(\SG^{\gr}(\mathscr{G}))$, for some $X \in \SG^{\gr}(\mathscr{G})$. As we showed above, $\mathcal{Y}_X = \{\mathcal{X}_x \mid x \in X\}$, and so $\mu_{\mathscr{G}}^{-1}(\mathcal{Y}_X) = X$, which is open, by hypothesis. It follows that $\mu_{\mathscr{G}}$ is continuous.

Finally, let $\phi : \mathscr{G} \to \mathscr{H}$ be an arbitrary morphism in $\GG$. Then an element $x \in \mathscr{G}$ is in the domain of $\phi$ (and $\mu_{\mathscr{H}} \circ \phi$) if and only if
\[\mathcal{X}_x \in \big\{\mathcal{X}_y \in \GP(\SG^{\gr}(\mathscr{G})) \mid (\phi^{-1})^{-1}(\mathcal{X}_y) \neq \emptyset\big\},\]
the domain of $(\phi^{-1})' : \GP(\SG^{\gr}(\mathscr{G})) \to \GP(\SG^{\gr}(\mathscr{H}))$ (and $\mu_{\mathscr{G}} \circ (\phi^{-1})'$)--see Lemmas~\ref{gr-inv-morph} and~\ref{gr-gp-morph}. It is now easy to see that the diagram below commutes.
\[\xymatrix{
\mathscr{G} \ar[rr]^-{\phi} \ar[d]_{\mu_{\mathscr{G}}} & & \mathscr{H} \ar[d]^{\mu_{\mathscr{H}}} \\
\GP(\SG^{\gr}(\mathscr{G})) \ar[rr]^-{(\phi^{-1})'} & & \GP(\SG^{\gr}(\mathscr{H}))}\]
Thus $\mu :1_{\GG} \rightarrow \GP\SG^{\gr}$ is a natural isomorphism.
\end{proof}

\begin{lemma} \label{sgp-nat-iso}
Defining a function 
\begin{align}
\nu_{S} : S & \longrightarrow \SG^{\gr}(\GP(S))  \notag \\
s &\longmapsto \mathcal{Y}_s  \notag
\end{align}
for each object $S$ in $\, \GIS$, gives a natural isomorphism $\nu :1_{\GIS} \rightarrow \SG^{\gr}\GP$, where $\, 1_{\GIS}$ denotes the identity functor on $\, \GIS$.
\end{lemma}

\begin{proof}
Let $S$ be an object in $\GIS$. Then, by Lemma~\ref{inv-slice}(4), $\nu_{S}$ is a semigroup homomorphism. By Remark~\ref{bool-rem2} and Lemma~\ref{inv-slice2}(2,6), $\nu_{S}$ is surjective. Lemma~\ref{inv-slice2}(1) implies that $\mathcal{Y}_s = \mathcal{Y}_t$ if and only if $s = t$, for all $s,t \in S$, from which it follows that $\nu_{S}$ is injective, and is hence a semigroup isomorphism. Therefore, by Lemma~\ref{inv-slice2}(1), $\nu_{S}$ is also an order isomorphism. In particular, $\nu_{S}$ preserves meets (see also Lemma~\ref{inv-slice}(1)), and restricts to a lattice homomorphism $E(S) \to E(\SG^{\gr}(\GP(S)))$ (see also Lemma~\ref{inv-slice}(1) and Lemma~\ref{inv-slice2}(2)), which preserves complements in principal order ideals. By Lemma~\ref{gr-gp-obj}, if $s \in S_{\alpha}$ for some $\alpha \in \Gamma$, then $\mathcal{Y}_s \in \SG^{\gr}(\GP(S))_{\alpha}$, and so $\nu_{S}$ is graded. Therefore $\nu_{S}$ is an isomorphism in $\GIS$.

Finally, for all morphisms $\phi : S \to T$ in $\GIS$, it is clear that the diagram below commutes (see Lemmas~\ref{gr-gp-morph} and~\ref{gr-inv-morph}).
\[\xymatrix{
S \ar[rr]^-{\phi} \ar[d]_{\nu_{S}} & & T \ar[d]^{\nu_{T}} \\
\SG^{\gr}(\GP(S)) \ar[rr]^-{(\phi')^{-1}} & & \SG^{\gr}(\GP(T))}\]
Thus $\nu :1_{\GIS} \rightarrow \SG^{\gr}\GP$ is a natural isomorphism.
\end{proof}

Combining the results above gives the following graded analogue of~\cite[Theorem 1.4]{Lawson3}.

\begin{thm} \label{gr-duality-thm}
The categories $\, \GIS$ and $\, \GG$ are dually equivalent, via the functors defined in Propositions~\ref{gp-to-sg-fuct} and~\ref{smgp-to-gp-top}, and natural transformations defined in Lemmas~\ref{gp-nat-iso} and~\ref{sgp-nat-iso}.
\end{thm}

Since our duality takes proper semigroup homomorphisms in $\GIS$ to fully-defined functors in $\GG$, and vice versa (see Lemmas~\ref{gr-gp-morph} and~\ref{gr-inv-morph}), we can recover Lawson's duality \cite[Theorem 1.4]{Lawson3} from the above result, by taking $\Gamma = \{\varepsilon\}$, and identifying $\GIS$ and $\GG$ with their non-graded counterparts.

We conclude this section with a couple of illustrations of Theorem~\ref{gr-duality-thm}.

\begin{example} \label{set-prod-groupoid}
For any set $X$, we can view $X \times X$ as a groupoid, with unit space $X$, and morphism set $X\times X$. Here $\domr((x,y)) := y$, $\ran((x,y)) := x$, and $(x,y)(y,z) := (x,z)$ for all $(x,y),(y,z) \in X \times X$.

According to~\cite[Example 2.25]{Lawson2}, for a finite nonempty set $X$, the discrete (Hausdorff ample) groupoid $X \times X$ corresponds to the symmetric inverse monoid $\mathcal{I}(X)$, in the Lawson-Stone duality (see~\cite{Lawson2} or~\cite[Theorem 1.4]{Lawson3}). More precisely, (compact) slices of $X\times X$ are subsets $\{(x_i,y_i) \mid i \in I\}$ of $X \times X$, such that $x_i\neq x_j$ and $y_i \neq y_j$, for distinct $i,j \in I$. Each slice $\{(x_i,y_i) \mid i \in I\}$ can be viewed as a partial symmetry $\phi: \{x_i\}_{i \in I} \to \{y_i\}_{i\in I}$ of $X$, with $\phi(x_i) = y_i$ for all $i \in I$, and the multiplication in $\SG(X\times X)$ corresponds precisely to that in $\mathcal{I}(X)$. Thus $\SG(X\times X)$ can be identified with $\mathcal{I}(X)$.

Now suppose that $X$ is a finite nonempty set with a grading $\phi: X \to \Gamma$, for some group $\Gamma$ (see Section~\ref{eg-sect} for the terminology). Then it is easy to see that $X\times X$ becomes a discrete $\Gamma$-graded groupoid, via $(x,y)\mapsto \phi(x)^{-1}\phi(y)$. In the duality of Theorem~\ref{gr-duality-thm}, $X\times X$ corresponds to $\mathcal{I}^{\gr}(X)$ (see Proposition~\ref {gr-sym-inv-semgp}). To see this, note that $\SG^{\gr}(X\times X)$ consists of slices $\{(x_i,y_i) \mid i \in I\}$ of $X \times X$, such that $\phi(x_i)^{-1}\phi(y_i) = \phi(x_j)^{-1}\phi(y_j)$ for all $i,j \in I$. It follows that each element $\{(x_i,y_i) \mid i \in I\}$ of $\SG^{\gr}(X\times X)$ corresponds to an element of $\mathcal{I}(X)_{\alpha}$, where $\alpha = \phi(x_i)^{-1}\phi(y_i)$ for all $i \in I$, in the above identification. So $\SG^{\gr}(X\times X) \cong \mathcal{I}^{\gr}(X)$ as $\Gamma$-graded inverse semigroups.
\end{example}

To give a sense for the possibilities afforded by using non-proper semigroup homomorphisms and partial functors in our duality, let us consider a more specialized example.

\begin{example} \label{morph-eg}
Let $X := \{x,y\}$, and view $\mathcal{I}(X) = \mathcal{I}^{\gr}(X)$ as a $\Z_2$-graded semigroup--see Example~\ref{symm-inv-eg}. Also letting $Y := \{x\} \subseteq X$, view $\mathcal{I}(Y) = \mathcal{I}^{\gr}(Y)$ as trivially-graded by $\Z_2$. Recall that given two partial symmetries $\phi, \psi$, we have  $\phi \leq \psi$ if and only if $\dom(\phi) \subseteq \dom(\psi)$, and $\phi(x)=\psi(x)$ for all $x \in \dom(\phi)$. So, clearly, the natural embedding $\iota : \mathcal{I}^{\gr}(Y) \to \mathcal{I}^{\gr}(X)$ is a morphism in the category $\Z_2\bf{GrIS}$. It is, however, not proper. To see this, consider $\{\theta_{xy},\tau\} \subseteq \mathcal{I}^{\gr}(X)_1$, using the notation of Example~\ref{symm-inv-eg}. Since $\theta_{xy} \leq \tau$, this set is downward directed, and also $\{\theta_{xy},\tau\} = \{\theta_{xy},\tau\}^{\uparrow}$. Since any filter containing $\{\theta_{xy},\tau\}$ must be a subset of $\mathcal{I}(X)_1 \setminus \{0\} = \{\tau, \theta_{xy}, \theta_{yx}\}$, by Lemma~\ref{filter-lemma}(3), and $\theta_{xy} \land \theta_{yx} = 0$, the set $\{\theta_{xy},\tau\}$ is an ultrafilter on $\mathcal{I}^{\gr}(X)$. However, $\iota^{-1}(\{\theta_{xy},\tau\}) = \emptyset$, and so $\iota$ is not proper.

As discussed in Example~\ref{set-prod-groupoid}, $\mathcal{I}^{\gr}(X)$ and $\mathcal{I}^{\gr}(Y)$ correspond (up to isomorphism), in the duality of Theorem~\ref{gr-duality-thm}, to the groupoids $X \times X$ and $Y \times Y$, respectively, with appropriate $\Z_2$-gradings. Using Lemma~\ref{gr-inv-morph}, it is not hard to see that $\iota$ corresponds in that duality to a partial functor $X \times X \to Y \times Y$, with domain $\{(x,x)\}$.
\end{example}

\section{Enveloping Rings} \label{rings-sect}

This final section is devoted to relating certain enveloping rings constructed from the semigroups $\SG(\mathscr{G})$ and $\SG^{\gr}(\mathscr{G})$, discussed in Section~\ref{duality-sect}. We begin by recalling graded semigroup rings.

Let $R$ be a (unital) ring, and $\Gamma$ a group. We say that $R$ is $\Gamma$-\emph{graded} if $R = \bigoplus_{\alpha \in \Gamma} R_\alpha$, where the $R_\alpha$ are additive subgroups of $R$, and $R_\alpha R_\beta \subseteq R_{\alpha \beta}$ for all $\alpha, \beta \in \Gamma$ (with $R_\alpha R_\beta$ denoting the set of all sums of elements of the form $rp$, for $r \in R_\alpha$ and $p \in R_\beta$).

Next, given a semigroup $S$, we denote by $R[S]$ the \emph{contracted} semigroup ring of $S$ (where the zero of $S$ is identified with the zero of $RS$). We denote an arbitrary element of $R[S]$ by $\sum_{s\in S} r^{(s)}s$, where $r^{(s)} \in R$, and all but finitely many of the $r^{(s)}$ are zero. If $S$ is $\Gamma$-graded, then it is easy to check that defining 
\[R[S]_{\alpha} := R[S_{\alpha}] = \Big\{\sum_{s\in S} r^{(s)}s \mid s \in S_\alpha \text{ whenever } r^{(s)}\neq 0\Big\},\]
for each $\alpha \in \Gamma$, turns $R[S]$ into a $\Gamma$-graded ring. See~\cite[Section 7]{RZ} for additional details.

Supposing that $S$ is a $\Gamma$-graded-Boolean inverse semigroup, we define the \emph{enveloping $R$-ring of $S$} (or \emph{$R$-algebra}, in case $R$ is commutative) as $R\langle S \rangle := R[S]/I$, where
\begin{equation*}
I := \langle u+v-u\vee v \mid u, v \in S_{\alpha} \text{ orthogonal, } \alpha \in \Gamma \rangle.
\end{equation*} 
Now, $I$ is a \emph{graded} ideal of $R[S]$, that is, $I = \bigoplus_{\alpha \in \Gamma}(I \cap R[S]_{\alpha})$. Hence the $\Gamma$-grading on $R[S]$ induces a \emph{quotient} $\Gamma$-grading on $R\langle S \rangle$, via 
\[R\langle S \rangle_{\alpha} := R[S]_{\alpha} + I.\]
See~\cite[Section 1.1.5]{hazi} for more on quotient gradings. The non-graded version of $R\langle S \rangle$ (i.e., where the grading on $S$, and hence also on $R\langle S \rangle$, is trivial) is discussed in~\cite[\S 6.3]{wehrung}.  

Given a $\Gamma$-graded Hausdorff ample groupoid $\mathscr{G}$, the graded inverse semigroup $\SG^{\gr}(\mathscr{G})$ is generally a proper subsemigroup of $\SG(\mathscr{G})$, whenever the grading on $\mathscr{G}$ is nontrivial (i.e., $\mathscr{G} \neq \mathscr{G}_{\varepsilon}$). Our next result shows, however, that those two semigroups produce isomorphic enveloping rings.

\begin{thm}
Let $R$ be a ring and $\mathscr{G}$ a $\, \Gamma$-graded Hausdorff ample groupoid. Then
\[R\langle \SG(\mathscr{G}) \rangle \cong R\langle \SG^{\gr}(\mathscr{G}) \rangle.\]
\end{thm}

\begin{proof}
We can define $\phi: R[\SG^{\gr}(\mathscr{G})] \rightarrow R[\SG(\mathscr{G})]$ by letting $\phi (X) = X$ for each $X \in \SG^{\gr}(\mathscr{G})$, and extending $R$-linearly to all of $R[\SG^{\gr}(\mathscr{G})]$. Since $\SG^{\gr}(\mathscr{G})$ is a subsemigroup of $\SG(\mathscr{G})$, the map $\phi$ is clearly an injective ring homomorphism. Since the join operation in both $\SG(\mathscr{G})$ and $\SG^{\gr}(\mathscr{G})$ is simply set union (see Lemma~\ref{gr-inv-obj}), $\phi$ reduces to a ring homomorphism $\phi: R\langle \SG^{\gr}(\mathscr{G}) \rangle \rightarrow R\langle \SG(\mathscr{G}) \rangle$.

Next, for each $X \in \SG(\mathscr{G})$ and $\alpha \in \Gamma$, let $X_{\alpha} :=X\cap \mathscr{G}_{\alpha}$. Then $X=\bigcup_{\alpha \in \Gamma} X_{\alpha}$, where the union is disjoint, and $X_{\alpha} \in \SG^{\gr}(\mathscr{G})$ for each $\alpha \in \Gamma$. Since $X$ is compact, only finitely many of the $X_{\alpha}$ are nonempty. Define $\psi: R[\SG(\mathscr{G})] \rightarrow R[\SG^{\gr}(\mathscr{G})]$ by letting $\psi(X) = \sum _{\alpha \in \Gamma} X_{\alpha}$ for each $X \in \SG(\mathscr{G})$, and extending $R$-linearly to all of $R[\SG(\mathscr{G})]$. We shall next show that $\psi$ reduces to a function $\psi: R\langle \SG(\mathscr{G}) \rangle \rightarrow R\langle \SG^{\gr}(\mathscr{G}) \rangle$.

Let $X,Y \in \SG(\mathscr{G})$ be orthogonal slices. If $\alpha, \beta \in \Gamma$ are distinct, then $X$ being a slice implies that $X_{\alpha}$ and $X_{\beta}$ have disjoint domain sets, and hence $X_{\alpha}^{-1}X_{\beta}=\emptyset$. Thus in $R\langle \SG(\mathscr{G}) \rangle$ we have 
\[\overline{X} = \bigcup_{\alpha \in \Gamma} \overline{X_{\alpha}} =\sum_{\alpha \in \Gamma} \overline{X_{\alpha}},\] where $\overline{X}$ denotes the image of $X$ in $R\langle \SG(\mathscr{G}) \rangle$, and likewise for $\overline{X_{\alpha}}$. The same considerations apply to $Y$ and $X\cup Y$. Therefore in $R\langle \SG(\mathscr{G}) \rangle$ we have
\[\overline{X}+ \overline{Y}-\overline{X\cup Y}=\sum_{\alpha\in \Gamma} \big(\overline{X_{\alpha}}+ \overline{Y_{\alpha}}-\overline{X_{\alpha}\cup Y_{\alpha}}\big).\] 
It follows that $\psi$ reduces to a well-defined function $\psi: R\langle \SG(\mathscr{G}) \rangle \rightarrow R\langle \SG^{\gr}(\mathscr{G}) \rangle$.

It is easy to see that both $\phi\psi: R\langle \SG(\mathscr{G}) \rangle \rightarrow R\langle \SG(\mathscr{G}) \rangle$ and $\psi\phi: R\langle \SG^{\gr}(\mathscr{G}) \rangle \rightarrow R\langle \SG^{\gr}(\mathscr{G}) \rangle$ are the identity function, making $\phi$ a bijection. Therefore $\phi$ is an isomorphism, as desired.
\end{proof}

\section*{Acknowledgements}

R.\ Hazrat was supported by Australian Research Council grant DP230103184. 

We would like to thank Benjamin Steinberg for a helpful discussion of this material, and Ganna Kudryavtseva for pointing us to relevant literature. 

We are grateful to the referee for numerous helpful suggestions, which have led to notable improvements in the paper. The most significant one was to allow in our inverse semigroup categories morphisms that are not necessarily proper, and, correspondingly, allow in our groupoid categories partial functors as morphisms.

\end{document}